\documentclass[11pt]{article}

\usepackage{graphicx,color,float,epstopdf}
\usepackage[linesnumbered,ruled]{algorithm2e}
\usepackage{multirow,array}
\usepackage{pst-blur,pstricks-add}
\usepackage{latexsym, bm, amsthm, amsmath,amssymb}
\usepackage{cases}
\usepackage{epsfig}
\usepackage{epstopdf}
\usepackage{contour,soul,color}
\usepackage[colorlinks=true]{hyperref}
\hypersetup{urlcolor=blue, citecolor=blue,linkcolor=blue}

\def \[{\begin{equation}}
\def \]{\end{equation}}
\def \dist{\hbox{dist}}

\def\B{{\cal B}}
\def\A{{\cal A}}
\def\L{{\cal L}}
\def\S{{\cal S}}
\def\E{{\cal E}}
\def\H{{\cal H}}
\def\T{{\cal T}}
\def\M{{\cal M}}
\def\bdes{\begin{description}}
\def\edes{\end{description}}
\def\benu{\begin{enumerate}}
\def\eenu{\end{enumerate}}
\def\bitm{\begin{itemize}}
\def\eitm{\end{itemize}}

\def\epi{{\rm epi }\,}

\def\L{{\cal L}}

\def\M{{\cal M}}

\def\U{{\cal U}}

\def\R{{\sl I\kern-3.2pt R}}

\def\r{\right}

\def\S{{\cal S}}

\def\sqr#1#2{{\vcenter{\hrule height .#2pt
      \hbox{\vrule width .#2pt height#1pt \kern#1pt\vrule width.#2pt}
                       \hrule height.#2pt}}}

\def\U{{\mathcal{U}}}
\def\Z{{\mathbf{Z}}}

\def\r1{{\mathbb{R}}}

\def\Z{{\cal Z}}

\def\U{{\cal U}}
\def\W{{\cal W}}

\newtheorem{lemma}{Lemma}[section]
\newtheorem{proposition}{Proposition}[section]
\newtheorem{corollary}{Corollary}[section]
\newtheorem{theorem}{Theorem}[section]

\newtheorem{definition}{Definition}[section]

\newtheorem{thm}{Theorem}[section]

\newtheorem{ass}[thm]{Assumption}

\def\X{\mathcal{X}}
\def\Y{\mathcal{Y}}
\def\Z{\mathcal{Z}}

\begin{document}
\graphicspath{{FIG2/}}

\vskip -1.5cm
\title{Linear   Rate Convergence of  the Alternating Direction Method of Multipliers for Convex Composite  Quadratic and Semi-Definite   Programming}
\author{Deren Han\thanks{School of Mathematical Sciences, Key Laboratory for NSLSCS of Jiangsu Province,
Nanjing Normal University, Nanjing 210023,  China. Email: \texttt{handeren@njnu.edu.cn}. The research of this author was supported
by the NSFC grants 11371197 and  11431002},  \/ \hskip 0.3cm
Defeng Sun\thanks{Department of Mathematics and Risk Management Institute, National University of Singapore, 10 Lower Kent Ridge Road, Singapore. Email: \texttt{matsundf@nus.edu.sg}. The research of this author was supported in part    by the Academic
Research Fund (Grant No. R-146-000-207-112). }\ \  \ \ and
\/ \hskip 0.3cm Liwei Zhang\thanks{{School of Mathematical Sciences, Dalian University of Technology, China. Email:  \texttt{lwzhang@dlut.edu.cn}. The research of this author was supported  by the National Natural Science Foundation of China under project
No. 91330206.}}
}

\date{August 10, 2015}

\maketitle

\begin{abstract}
In this paper, we aim to prove   the   linear  rate   convergence of the  alternating direction method of multipliers (ADMM) for solving linearly constrained convex composite optimization problems.
Under an error bound condition, we establish  the global linear rate of convergence for   a more general  semi-proximal  ADMM with the dual steplength being restricted to be in   $( 0, (1+\sqrt{5})/2)$. In our analysis,   we assume    neither the  strong convexity nor the strict complementarity except the  error bound condition, which holds automatically for convex composite quadratic programming.  This semi-proximal ADMM, which  covers  the classic  one,    has the advantage to
   resolve the potentially non-solvability issue of the subproblems in the classic   ADMM and  possesses the abilities of handling  the  multi-block cases efficiently. We shall use  convex composite quadratic  programming  and quadratic semi-definite programming   to demonstrate the significance of the obtained results.  Of its own novelty in second-order variational analysis,  a complete characterization is provided on the isolated calmness for the    convex semi-definite optimization problem in terms of its second order sufficient optimality condition and the strict Robinson constraint qualification   for the purpose of  proving the linear rate convergence of the semi-proximal ADMM when applied to two- and  multi-block convex quadratic semi-definite programming.
\end{abstract}

\bigskip\noindent
{\bf Key words}: ADMM, error bound,  global linear rate, isolated calmness, composite quadratic programming, semi-definite optimization.


\section{Introduction}
In this paper, we shall study  the linear rate convergence of the alternating direction method of multipliers (ADMM) for solving  the following convex composite   optimization problem
\[\label{ad}
\min\;\{\vartheta (y)+g(y)+\varphi(z)+h(z): \,  \A^*y+\B^*z=c,\; y\in\Y,\; z\in\Z\},\]
where  $\mathcal{Y}$ and $\mathcal{Z}$ are two finite-dimensional real Euclidean spaces
each equipped with an inner product $\langle\cdot,\cdot\rangle$ and its induced norm $\|\cdot\|$,
$\vartheta:\mathcal{Y}\rightarrow(-\infty, +\infty]$ and $\varphi:\mathcal{Z}\rightarrow(-\infty, +\infty]$  are two
  proper closed convex functions,  $g:\mathcal{Y}\rightarrow(-\infty, +\infty)$ and $h:\mathcal{Z}\rightarrow(-\infty, +\infty)$ are two
continuously differentiable convex functions (e.g., convex quadratic functions),  $\A^*: \mathcal{Y}\to\mathcal{X}$ and $\B^*: \mathcal{Z}\to\mathcal{X}$ are the adjoints of the two linear operators $\A: \mathcal{X}\to\mathcal{Y}$ and $\B:  \mathcal{X}\to\mathcal{Z}$, respectively, with  $\mathcal{X}$ being another  real finite-dimensional Euclidean space
  equipped with an inner product $\langle\cdot,\cdot\rangle$ and its induced norm $\|\cdot\|$ and
 $c\in \mathcal{X}$ is a given point. For any convex function   $\theta: \X \to (-\infty, \infty]$, we use ${\rm dom\,} \theta$  to define its effective domain, i.e., ${\rm dom\,}\theta: =\{x\in \X:\, \theta(x)<\infty\}$,  $\epi \theta$ to denote its epigraph, i.e.,  ${\rm epi\,}\theta: =\{(x,t)\in \X\times \Re:\, \theta(x)\le t\}$ and $\theta^*:\X \to (-\infty, \infty]$ to represent its Fenchel conjugate, respectively.

The classic  ADMM  was   designed   by Glowinski and Marroco \cite{glo75} and  Gabay and Mercier \cite{Gabay:1976ff} and its construction  was much influenced by Rockafellar's works on proximal point algorithms (PPAs)  for solving the more general  maximal monotone inclusion problems \cite{rockafellar76A,rockafellar76B}. The readers may refer to Glowinski \cite{globook14} for a note on the historical development of the classic ADMM.
The  convergence analysis for the classic ADMM   under certain settings was  first conducted  by  Gabay and Mercier \cite{Gabay:1976ff}, Glowinski \cite{glo80book} and Fortin and Glowinski \cite{fortin83}.
For a recent survey on this,  see
  \cite{eck12}.

Our focus of this paper is on the linear rate convergence analysis of the ADMM. This shall be conducted   under a more convenient semi-proximal ADMM (in short, sPADMM) setting  proposed  by Fazel et al. \cite{Fazel} by allowing the dual step-length to be at least as large as  the golden ratio of $1.618$. This sPADMM, which  covers  the classic ADMM,  has the advantage to
   resolve the potentially non-solvability issue of the subproblems in the classic  ADMM. But, perhaps more importantly it possesses the abilities of handling   multi-block convex optimization problems. For example, it has been shown most recently that the sPADMM  plays a pivotal role in solving multi-block convex composite semi-definite programming problems \cite{STYang2015,LSToh2014, CSToh2015} of a low to medium accuracy.  We shall come back to this in Section \ref{Section:LinearRate}.

For any   self-adjoint   positive semi-definite linear operator $\M:\X\to \X$,  denote  $\|x\|_\M:=\sqrt{\langle x , \M x\rangle }$ and   $\dist_\M(x,C)=\inf_{x'\in C}\|x'-x\|_{\M}$ for any $x\in \X$ and any set $C\subseteq \X$. We use ${\cal I}$ to denote the identity mapping from $\X$ to itself. Let $\sigma >0$ be a given parameter. Write $\vartheta_g (\cdot)\equiv \vartheta(\cdot)+ g(\cdot)$ and $\varphi_h (\cdot)\equiv \varphi(\cdot)+ h(\cdot)$. The augmented Lagrangian function of problem (\ref{ad}) is defined by
\[\label{augL}
{\cal L}_{\sigma}(y,z;x):=\vartheta_g (y)+\varphi_h (z)+\langle x,\A^*y+\B^*z-c\rangle+\displaystyle \frac{\sigma}{2}\|\A^*y+\B^*z-c\|^2,\ \forall\, (y,z,x)\in \Y\times \Z\times \X.
\]
 Then the    sPADMM may be described as follows.

\fbox{
 \begin{minipage}[t]{0.9\textwidth}{\noindent\bf sPADMM}: A semi-proximal alternating direction method of multipliers for solving the convex optimization problem (\ref{ad}).

\begin{itemize}
\item[ ]\hskip -.5cm{\bf Step 0.} Input $(y^0,z^0,x^0)\in\hbox{dom}\; \vartheta\times \hbox{dom}\; \varphi\times \mathcal{X}.$ Let  $\tau\in (0,\infty)$ be a positive parameter (e.g., $\tau \in ( 0, (1+\sqrt{5})/2)$\,),  and  $\S:\Y\to \Y$ and $\T:\Z\to \Z$ be two self-adjoint positive semi-definite, not necessarily positive definite, linear operators. Set $k:=0$.
\item[ ]\hskip -.5cm{\bf Step 1.} Set \begin{subnumcases}{\hskip-1cm}\label{xna}
y^{k+1}\in \hbox{arg}\min \, \L_{\sigma }(y,z^k; x^k) +\frac{1}{2}\|y-y^k\|^2_\S\, ,\\
\label{xnb}z^{k+1}\in \hbox{arg}\min\, \L_{\sigma  }(y^{k+1},z; x^k) +\frac{1}{2}\|z-z^k\|^2_\T\, ,\\
\label{xnc}x^{k+1}=x^k+\tau\sigma (\A^*y^{k+1}+\B^*z^{k+1}-c).
\end{subnumcases}
\item[ ]\hskip -.5cm{\bf Step 2.} If a termination criterion is not met, set $k:=k+1$ and go to Step 1.
\end{itemize}

 \end{minipage} }

\medskip
The sPADMM scheme (\ref{xna})--(\ref{xnc})  with $\S=0$ and  $\T=0$  is nothing but the classic ADMM  of  Glowinski and Marroco \cite{glo75} and  Gabay and Mercier \cite{Gabay:1976ff}.
 When $\B ={\cal I}$ and $\A$ is surjective,   the global convergence of the classic ADMM with  any $\tau \in ( 0, (1+\sqrt{5})/2)$ has been established by  Glowinski \cite{{glo80book}} and Fortin and Glowinski \cite{fortin83}.
  Interestingly,  in \cite{Gabay1983}, Gabay
  has further shown that  the  classic  ADMM with $\tau =1$, under the  existence condition of a solution to the Karush-Kuhn-Tucker (KKT) system of problem (\ref{ad}),
  is actually equivalent to the  Douglas-Rachford (DR) splitting method applied to a stationary system to  the  dual of problem (\ref{ad}).  Moreover, Eckstein and Bertsekas \cite{EcksteinB1992} have
proven that the DR splitting method  can be equivalently represented as a special PPA. Thus, one may always use known results on the DR splitting method and the PPA  to study the properties of the classic ADMM with $\tau =1$ (this does not apply to the case that  $\tau \ne 1$ of course) though the corresponding  transformations can be much involved.
 The  above sPADMM scheme (\ref{xna})--(\ref{xnc}) with $\S\succ0$ and $\T\succ 0$ was initiated by Eckstein \cite{eckstein1994} to make the subproblems in
(\ref{xna}) and (\ref{xnb}) easier to solve. Using essentially the same variational   techniques   developed by Glowinski \cite{glo80book} and Fortin and Glowinski \cite{fortin83}, Fazel et al. developed   an extremely easy-to-use convergence theorem
 for the sPADMM   \cite[Appendix B]{Fazel} when the dual step-length $\tau$ is chosen to be in $ ( 0, (1+\sqrt{5})/2)$.   In \cite{STeboulle2014},  Shefi and Teboulle  conducted   a comprehensive  study on the iteration  complexities, in particular in the ergodic sense,  for the sPADMM with   $\tau =1$ and $\B\equiv {\cal I}$. Related results for the more general cases can be found, e.g.,  in \cite{LSToh2014M} for the case that  the linear operators $\S$ and $\T$ are allowed to be indefinite and in \cite{CLSToh2015} for the case that  the objective function is allowed to have a coupled smooth term.
 For  details on choosing $\S$ and $\T$, one may refer to the recent PhD thesis of Li \cite{{LiXDThesis2015}}.

Compared with the large amount of literature\footnote{For example,  according to Google Scholar, the survey paper by Boyd et al.  \cite{boyd11} on the applications of the ADMM with $\tau =1$ has been cited more than 2,289 times as of August 2, 2015.} mainly being devoted to the applications of the ADMM, there is a much   smaller  number of papers targeting the linear rate convergence analysis though there do exist a number of classic results and several  interesting new advancements   on the latter. By using the  aforementioned connections among the DR splitting method, PPAs, and the classic ADMM with $\tau =1$, we can derive the corresponding linear rate convergence of the ADMM from the works of
  Lions and Mercier \cite{LMercier1979} on  the DR splitting method with a globally Lipschitz continuous and strongly monotone operator and   Rockafellar \cite{rockafellar76A,rockafellar76B} and Luque \cite{Luque1984} on the convergence rates of the PPAs under  various error bound conditions imposed on the inverse of   maximal monotone operators. For  example, within this spirit,  Eckstein and Bertsekas \cite{EcksteinB1990}
proved    the global linear convergence rate of the ADMM with $\tau =1$ when it is applied to linear programming by using the equivalence of the ADMM and  a PPA.
For recent new developments on the linear convergence rate of the ADMM, we can  roughly categorize them  into the following three cases:
\begin{description}
\item[(i)] For   convex quadratic programming,  Boley \cite{Boley2013} provided a local linear convergence result for  the
ADMM with $\tau =1$  under the conditions of  the uniqueness of the optimal solutions to both the primal and dual problems and the strict complementarity; in  \cite{HYuan2013}, Han and Yuan removed the restrictive conditions imposed by Boley and established the local  linear rate convergence of the generalized  ADMM in the sense of Eckstein and Bertsekas \cite{EcksteinB1992} for the subsequence $\{(z^k, x^k)\}$; and in \cite{YHan2014}, Yang and Han  showed that the local linear rate result in \cite{HYuan2013} can be globalized under a slightly more general setting for  the
ADMM with $\tau =1$ and a linearized ADMM (a special case of sPADMM with $\S\succ 0$ and $\T\succ 0$) with $\tau =1$, where for the latter the linear rate is established for the whole sequence $\{(y^k,z^k, x^k)\}$ instead of only the subsequence $\{(z^k, x^k)\}$. We remark that when either $\S\succ 0$ or $\T\succ 0$ fails to hold, the linear rate convergence  analysis in \cite{YHan2014} is no longer valid.

\item[(ii)] In \cite{DYin2015}, Deng and Yin provided a number of scenarios  on  the linear rate convergence  for the ADMM and sPADMM with $\tau =1$ under the assumption that either $\vartheta_g (\cdot )$ or $\varphi_h (\cdot )$ is strongly convex with a Lipschitz continuous gradient in addition to the boundedness condition on the generated iteration sequence and others. Deng and Yin's focus is mainly on problems being reformulated from unconstrained composite models with applications in sparse optimization, e.g.,  the models of Lasso  regularized with strongly convex terms. They also made a detailed comparison between their most notable   linear rate convergence result and that   of Lions and Mercier \cite{LMercier1979} on the  DR splitting method   when   applied to a stationary system to the  dual of problem (\ref{ad}).

\item[(iii)]
Assuming an  error bound condition and some others, Hong and Luo \cite{HLuo2012}  provided  a linear rate convergence of  the multi-block  ADMM   with a sufficiently small
step-length $\tau$. Theoretically, this constitutes important progress  on understanding the convergence and the linear rate of convergence of the ADMM. Computationally, however, this is far from being satisfactory as  in practical implementations one always prefers a larger step-length for achieving numerical efficiency.
\end{description}

In this paper, we aim to resolve the linear rate convergence issue for the sPADMM scheme  (\ref{xna})--(\ref{xnc})  with  $\tau \in ( 0, (1+\sqrt{5})/2)$ assuming neither the  strong convexity for $\vartheta_g (\cdot)$ or $\varphi_h (\cdot)$ nor the strict complementarity. Special attention shall be paid  to convex composite    quadratic  programming and quadratic semi-definite programming. For the former, we have a complete picture  and for the latter we show how far we have progressed. More specifically,
our   main  contributions made in this paper include but are not limited to:
 \begin{description}
\item [(1)] Under an error bound condition only, we provide a very general linear rate convergence analysis for the sPADMM with $\tau \in ( 0, (1+\sqrt{5})/2)$. This is made possible by constructing     an elegant inequality on the iteration sequence   via re-organizing the relevant results developed in \cite[Appendix B]{Fazel}.
\item[(2)] For convex composite quadratic programming,  the global linear convergence rate is obtained with no additional conditions  as the error bound assumption holds automatically. By choosing the positive semi-definite linear operators  $\S$ and $\T$ properly, in particular $\T=0$, we demonstrate how the established global linear rate convergence  of the sPADMM can be  applied to   multi-block convex composite quadratic conic programming.

\item[(3)] For convex composite  quadratic semi-definite programming (SDP), a linear convergence rate is established under the assumption that both the primal and dual problems satisfy the second order sufficient optimality  condition, one of eight equivalent conditions proven in this paper. This is achieved via characterizing the isolated calmness of the corresponding  optimality systems.
 \item[(4)]  The obtained  results on the isolated calmness for convex and non-convex semi-definite optimization problems are not only important for the linear rate convergence analysis of the sPADMM but also are interesting in their own right in the context of sensitivity analysis for optimization problems with non-polyhedral cone constraints.
\end{description}

The remaining parts of the this paper are organized as follows. In Section \ref{preliminary}, we conduct brief discussions on  the optimality conditions for  problem (\ref{ad}) and on both the calmness and isolated calmness for multi-valued mappings.  Section \ref{Section:LinearRate} is divided into three parts with the first part focusing on deriving a particularly useful inequality for the iteration sequence generated  from the sPADMM. This inequality, which  grows out of the results in \cite[Appendix B]{Fazel}, is then employed to build up a general linear rate convergence theorem under an error bound condition. The third part of this section is about the applications of the linear convergence theorem of the sPADMM to important convex composite quadratic conic programming. Section \ref{section:isolatedcalmness} is devoted to the characterization of the isolated calmness for composite semi-definite optimization problems, which are not necessarily convex. The sufficient conditions for non-convex semi-definite optimization problems,  which are strongly motivated by the work done in \cite{Sun2006} on Robinson's strong regularity,   can be regarded as natural extensions to those established by Zhang and Zhang \cite{ZZhangLW2015}. The complete characterization of the isolated calmness in the convex case represents a significant step forward in second order variational analysis on  convex optimization problems constrained with non-polyhedral convex cones. In Section \ref{section:QSDP}, for convex composite quadratic semi-definite programming, we provide further deep results on the isolated calmness by relating the second order sufficient optimality condition for the primal problem equivalently to the strict Robinson constraint qualification for the corresponding dual problem.  We make our final conclusions in Section \ref{section:final}.

 \section{Preliminaries}\label{preliminary}
 \setcounter{equation}{0}
 In this section, we summarize some useful preliminaries for subsequent  analysis.

 \subsection{Optimality conditions}

For a multifunction $F:\mathcal{Y}\rightrightarrows\mathcal{Y}$, we say that $F$ is monotone if
\begin{eqnarray}\label{mono}
\langle y'-y, \xi^\prime -\xi\rangle\geq0,\qquad\forall\,  \xi^\prime\in F(y'),~\forall\, \xi\in F(y).
\end{eqnarray}
It is well known that for any proper closed convex function $\theta:\X\to (-\infty, \infty]$,  $\partial \theta(\cdot)$ is a monotone multi-valued function (see \cite{Roc70}), that is,
for any $w_1\in\hbox{dom}\; \theta$ and any $w_2\in\hbox{dom}\; \theta$,
\begin{eqnarray}\label{chui}
\langle \xi-\zeta, w_1-w_2\rangle\geq 0,\qquad\forall \, \xi\in\partial \theta(w_1),~\forall\, \zeta\in\partial\theta(w_2).
\end{eqnarray}

In our analysis, we shall often use the optimality conditions for problem (\ref{ad}). Let $(\bar{y},\bar{z})\in {\rm  dom}( \vartheta) \times {\rm dom}(\varphi)$ be an optimal solution to problem \eqref{ad}. If there exists $\bar{x}\in\mathcal{X}$ such that $( \bar{y}, \bar{z}, \bar{x})$ satisfies the following KKT  system
\begin{eqnarray}\label{KKT}
\left\{\begin{array}{l}
0\in  \partial \vartheta (y )+\nabla g(y )+\A x,\\
0\in  \partial \varphi (z )+\nabla h(z )+\B x,\\
c-\A^* y-\B^* z=0,
\end{array}\right.
\end{eqnarray}
then $( \bar{y}, \bar{z}, \bar{x})$ is called a KKT point  for problem \eqref{ad}. Denote  the solution  set to the KKT system   \eqref{KKT} by $\overline{\Omega}$.
The existence of such   KKT points can be guaranteed if a certain constraint qualification  such as the Slater condition holds:
$$
\exists \ (y',z')\in {\rm ri}  \left (\hbox{dom} (\vartheta)\times  \hbox{dom} (\varphi)\right )\cap  \{(y,z)\in {\cal Y}\times {\cal Z}:\, \A^*y+\B^*z=c\},
$$
where ${\rm ri}(S)$ denotes the relative interior of a given convex set $S$.
 In this paper, instead of using an explicit constraint qualification, we   make the following blanket assumption on the existence of a KKT point.
 \medskip
 \begin{ass}\label{ass.1}
  The KKT system (\ref{KKT})  has a non-empty solution set.
 \end{ass}
 \medskip

Denote $u:=(y,z,x)$ for $y \in {\cal Y}$, $z \in {\cal Z}$ and  $x \in {\cal X}$. Let $\mathcal{U}:=\Y\times\mathcal{Z}\times \X$. Define the KKT mapping  $R:\mathcal{U}\rightarrow\mathcal{U}$ as
\[\label{F}
R(u):= \left(\begin{array}{c}
 y- {\rm Pr}_\vartheta[y-(\nabla g(y)+\A x)]\\
z-{\rm Pr}_\varphi[z-(\nabla h(z)+\B x)]\\
c-\A^* y-\B^* z
\end{array}\right),\quad \forall\, u\in\mathcal{U},
\]
where for any convex function $\theta:\X \to (-\infty,\infty]$, ${\rm Pr}_\theta(\cdot) $ denotes its associated Moreau-Yosida  proximal mapping. If $\theta (\cdot) =\delta _{{\cal K}} (\cdot)$, the indicator function over the closed convex set ${\cal K}\subseteq \X$, then
 ${\rm Pr}_\theta(\cdot)=\Pi_{\cal K}(\cdot)$, the metric projection operator over ${\cal K}$.
Then,    since  the Moreau-Yosida proximal mappings ${\rm Pr}_\vartheta(\cdot) $ and  ${\rm Pr}_\varphi(\cdot) $ are  both globally Lipschitz continuous with modulus one, the mapping $R(\cdot)$ is at least continuous on $\U$ and
$$
\forall\, u\in {\cal U}, \quad R(u)=0 \Longleftrightarrow u\in \overline{\Omega}.
$$

\subsection{Calmness and isolated calmness}
Let  ${\cal X}$ and ${\cal Y}$ be two finite-dimensional real Euclidean  spaces and $F:{\cal X}\rightrightarrows {\cal Y}$ be a set-valued mapping with $(x^0,y^0) \in {\rm gph}\,F$, the graph of $F$. Let  $\textbf{B}_{{\cal Y}}$ denote  the unit ball in  ${\cal Y}$.
 \begin{definition}\label{def:upperLipschitz}
   The  multi-valued mapping  $F:{\cal X}\rightrightarrows {\cal Y}$ is said to be calm  at  $x^0$  if there is  a constant $\kappa_0>0$ along with  a neighborhood $V$ of $x^0$   such that
$$
  F(x)  \subseteq F(x^0)+\kappa_0 \|x-x^0\| \textbf{B}_{{\cal Y}}, \quad  \forall\,  x \in V.
$$
\end{definition}

The above definition of calmness for the multi-valued mapping $F$ comes from \cite[9(30)]{RW98} and it was called the upper Lipschitz continuity in \cite{Robinson1979}. Recall that    the multi-valued mapping $F$ is called piecewise polyhedral if ${\rm gph}\,F$ is the union of finitely many polyhedral sets. In one of his landmark papers, Robinson \cite{Robinson1981} established the following  important   property on the calmness for a piecewise polyhedral multi-valued mapping.

\begin{proposition}\label{prop:piecewisePoly} If the multi-valued mapping $F:{\cal X}\rightrightarrows {\cal Y}$ is  piecewise polyhedral, then $F$ is calm at $x^0$.
\end{proposition}

 Next, we   give the definition of isolated calmness for   $F: {\cal X}\rightrightarrows {\cal Y}$  at  $x^0$ for $y^0$.
\begin{definition}\label{def:calm}
   The multi-valued mapping  $F:{\cal X}\rightrightarrows {\cal Y}$ is said to be isolated  calm at  $x^0$ for $y^0$  if there is a constant $\kappa_0>0$ along with  a neighborhood $V$ of $x^0$ and a neighborhood $W$ of $y^0$  such that
$$
  F(x) \cap W \subseteq \{y^0\}+\kappa_0 \|x-x^0\| \textbf{B}_{{\cal Y}}, \quad  \forall\,  x \in V.
$$
\end{definition}
The isolated calmness given in Definition \ref{def:calm} was  called differently in the literature, e.g., the    local upper Lipschitz continuity in \cite{Dontchev1995,Levy96}, to distinguish it from Robinson's definition of upper Lipschitz continuity \cite{Robinson1979}. Here we adopt the usage in \cite{DR2001,Dontchev2009}.
The concept of graphical derivative  of $F$ \cite[8.33 Definition]{RW98} is a convenient tool for investigating the isolated calmness property.  The graphical derivative of $F$
at $x^0$ for $y^0\in F(x^0)$ is the set-valued mapping $D F(x^0| y^0):{\cal X}  \rightrightarrows {\cal Y}$  whose graph is the tangent cone
$T_{{\rm gph}\, F} (x^0,y^0)$, namely for any  $(u,v) \in \X\times \Y$,
$$
v \in DF(x^0|y^0)(u) \Longleftrightarrow (u,v)\in T_{{\rm gph}\, F} (x^0,y^0).
$$
 In other words,
$v \in D F(x^0|y^0)(u)$ if and only if
$$
\left
\{
\begin{array}{l}
{\rm there\,\, exist\,\, sequences}\,\,\, t_k \rightarrow 0_+, u^k \rightarrow u \mbox{ and  }v^k \rightarrow v\\[4pt]
{\rm such\,\, that}\,\,\, v^k \in \displaystyle \frac{F(x^0 +t_ku^k)-y^0}{t_k} \mbox{ for all } k.
\end{array}
\right.
$$
It follows from    \cite[8(19)]{RW98} that the following equivalence holds:
\begin{equation}\label{inv-D}
v \in DF(x^0|y^0)(u) \Longleftrightarrow u \in D(F^{-1})(y^0|x^0)(v).
\end{equation}

A basic characterization of the isolated calmness property for a set-valued mapping at a point  is given by the following lemma.
\begin{lemma}\label{lem.1.1}
(King and Rockafellar \cite{KR92}, Levy \cite{Levy96}) Let $(x^0,y^0)\in {\rm gph}\, F$. Then  $F$ is isolated calm at  $x^0$ for $y^0$ if and only if $\{0\} = DF(x^0|y^0)(0)$.
\end{lemma}

\section{A general theorem on the linear rate convergence}\label{Section:LinearRate}
\setcounter{equation}{0}

In this section, we shall establish a general theorem on the linear convergence rate of the sPADMM scheme \eqref{xna}-\eqref{xnc}.


First we recall the global convergence of the sPADMM from \cite[Appendix B]{Fazel}.
 Since both  $\partial \vartheta$ and $\partial \varphi$ are maximally monotone, there exist two self-adjoint and positive semi-definite linear operators
$\Sigma_{\vartheta_g}$ and $\Sigma_{\varphi_h}$ such that for all $y',y \in {\rm dom}\,\vartheta_g$, $\xi \in \partial \vartheta_g(y)$ and $\xi'\in \partial \vartheta_g(y')$, and for all $z',z\in {\rm dom}\,\varphi_h$, $\zeta \in \partial \varphi_h(z)$ and $\zeta'\in \partial \varphi_h(z')$,
\begin{equation}\label{eq:B.4}
\langle \xi'-\xi,y'-y \rangle \geq \|y'-y\|^2_{\Sigma_{\vartheta_g}},\,\, \ \langle \zeta'-\zeta, z'-z \rangle \geq \|z'-z\|^2_{\Sigma_{\varphi_h}}.
\end{equation}

For notational convenience, let $\E: {\cal X} \rightarrow {\cal U}:={\cal Y}\times {\cal Z}\times \X$ be a linear operator  such that its adjoint $\E^*$ satisfies   $\E^*(y,z,x)=\A^*y+\B^*z$ for any $(y,z,x)\in {\cal Y}\times {\cal Z}\times \X$ and for  $u:=(y,z,x)\in {\cal U}$ and $u':=(y',z',x')\in {\cal U}$, define
$$
\theta (u,u'):=(\tau\sigma )^{-1}\|x-x'\|^2+\|y-y'\|_{\S}^2+\|z-z'\|_{\T}^2+\sigma \|\B^*(z-z')\|^2.
$$
The following theorem, which will be used in the following, is  adapted  from Appendix B of \cite{Fazel}.
\begin{theorem}\label{Th:B1}
 Let Assumption  \ref{ass.1}   be satisfied.   Suppose that the sPADMM generates a well defined infinite sequence $\{u^k\}$. Let  $\bar u=(\bar y,\bar z,\bar x) \in \overline{\Omega}$.  For $k\ge 1,$   denote
\begin{equation}\label{eq:B.7}
\left
\{
\begin{array}{l}
\delta_{k}:=\tau (1-\tau+\min \{\tau,\tau^{-1}\}) \sigma  \|\B^*(z^{k}-z^{k-1})\|^2+\|z^{k}-z^{k-1}\|^2_\T,\\[4pt]
\nu_{k}:=\delta_{k}+\|y^{k}-y^{k-1}\|^2_\S+2\|y^{k}- \bar y\|^2_{\Sigma_{\vartheta_g}}+2\|z^{k}-\bar z\|^2_{\Sigma_{\varphi_h}}.
\end{array}
\right.
\end{equation}
Then, the following results hold:
\begin{description}
\item[(i)] For any $k\geq 1$,
\begin{equation}\label{eq:B.8}
\begin{array}{l}
\left [\theta(u^{k+1},\overline u)+\|z^{k+1}-z^k\|^2_\T+\displaystyle (1-\min\{\tau,\tau^{-1}\})\sigma \|\E^*(y^{k+1},z^{k+1},0 )-c\|^2\right]\\[16pt]
 -\left[\theta(u^{k},\bar u)+\|z^k-z^{k-1}\|^2_\T+\displaystyle (1-\min\{\tau,\tau^{-1}\})\sigma  \|\E^*(y^{k},z^{k},0)-c\|^2\right]\\[16pt]
\leq -\left[\nu_{k+1}+\displaystyle (1-\tau+\min\{\tau,\tau^{-1}\})\sigma \|\E^*(y^{k+1},z^{k+1},0)-c\|^2\right].
\end{array}
\end{equation}
\item[(ii)] Assume  that both $ \Sigma_{\vartheta_g} +\S+ \sigma \A\A^*$ and $ \Sigma_{\varphi_h} +\T+ \sigma \B\B^*$ are positive definite so that the sequence    $\{u^k\}$ is automatically well defined.  If $\tau \in (0, (1+\sqrt{5})/2)$, then the whole sequence $\{(y^k,z^k,x^k)\}$ converges to     a KKT point in $\overline{\Omega}$.
    \end{description}
\end{theorem}

\medskip
For any self-adjoint linear operator $\M: \X\to \X$, we use $\lambda_{\max}(\M)$ to denote its largest eigen-value.
Define
$\kappa:=\max\left\{\kappa_1,\kappa_2,\kappa_3\right\}$,
where $$\kappa_1: =3\|\S\|,\;\; \kappa_2:  =\max\{   3 \sigma \lambda_{\max} (\A \A^*), 2 \|\T \|\}$$
and
$$\kappa_3:=3{(1-\tau)^2}{\sigma }\lambda_{\max}(\A\A^*)+2{(1-\tau)^2}{\sigma }\lambda_{\max}(\B\B^*)+ {\sigma ^{-1}}.$$
Let
\[\label{barH}
  \H_0:=\kappa \,{\rm Diag}\left (\S, \T+\sigma \B\B^*, (\tau^2\sigma )^{-1}{\cal I}\right)
\]
be a block-diagonal positive semi-definite linear operator from $ \Y\times \Z\times \X$ to itself such that
$$
 \H_0(y,z,x) ={ {\kappa}}\left (\S y,\, (\T+\sigma \B\B^*)z,\, (\tau^2\sigma )^{-1}x\right), \quad \forall\, (y,z,x)\in \Y\times \Z\times \X.
$$
\begin{lemma}\label{jiu}
Let $\{u^k:=(y^k,z^k,x^k)\}$ be the infinite sequence generated by the sPADMM scheme \eqref{xna}-\eqref{xnc}.
Then for any $k\ge 0$,
\begin{eqnarray}\label{cou}
\|u^{k+1}-u^k\|^2_{ \H_0}\geq \|R(u^{k+1})\|^2.
\end{eqnarray}

\end{lemma}

\begin{proof}The optimality condition for \eqref{xna} is
\[\label{ox}
 0\in\partial \vartheta (y^{k+1})+\nabla g(y^{k+1})+\A[x^k+\sigma (\A^*y^{k+1}+\B^*z^{k}-c)]+\S(y^{k+1}-y^k).\]
From the definition of $x^{k+1}$, we have
$$x^k+\sigma (\A^*y^{k+1}+\B^*z^k-c)=-\sigma  \B^*(z^{k+1}-z^k)+x^k+\tau^{-1}(x^{k+1}-x^k).$$
It then follows from \eqref{ox}  that
\begin{eqnarray}\label{ox2}
 0&\in&\partial \vartheta (y^{k+1})+\nabla g(y^{k+1})+\A[x^k+\sigma (\A^*y^{k+1}+\B^*z^{k}-c)]+\S(y^{k+1}-y^k)\nonumber\\
 &=& \partial \vartheta (y^{k+1})+\nabla g(y^{k+1})+\A[\sigma \B^*(z^k-z^{k+1})+x^k+\tau^{-1}(x^{k+1}- x^k)]+\S(y^{k+1}-y^k),\nonumber\end{eqnarray}
 which implies
 \begin{equation}\label{ox3}
 \begin{array}{l}
y^{k+1}=\\[4pt]
{\rm Pr}_\vartheta\left(y^{k+1}-\left(\nabla g(y^{k+1})+\A[\sigma  \B^*(z^k-z^{k+1})+x^k+\tau^{-1}(x^{k+1}- x^k)]
+\S(y^{k+1}-y^k)\right)\right).
\end{array}
\end{equation}
Noting that since $z^{k+1}$ is a solution to the subproblem \eqref{xnb},
we have that
$$0\in\partial \varphi (z^{k+1})+\nabla h(z^{k+1})+\B x^k+\sigma  \B(\A^*y^{k+1}+\B^*z^{k+1}-c)+\T(z^{k+1}-z^k),
$$
which is equivalent to
$$
0\in\partial \varphi(z^{k+1})+\nabla h(z^{k+1})+\B[x^k+\tau^{-1}(x^{k+1}-x^k)]+\T(z^{k+1}-z^k).
$$
Thus, we have
\[\label{suby2}
z^{k+1}={\rm Pr}_\varphi\left (z^{k+1}-\left(\nabla h(z^{k+1})+\B[x^k+\tau^{-1}(x^{k+1}-x^k)]+\T(z^{k+1}-z^k)\right)\right).
\]
Note that from \eqref{xnc},
\[\label{suby3}
x^{k+1}=x^{k}+\tau \sigma  (\A^*y^{k+1}+\B^*z^{k+1}-c).
\]
Then, by  coming (\ref{ox3}), (\ref{suby2}) and (\ref{suby3}) and  noticing of the Lipschitz continuity of  the Moreau-Yosida proximal mappings, we obtain from the definition of $R(\cdot)$ in $(\ref{F})$ that
$$
\begin{array}{ll}&\quad \|R(u^{k+1})\|^2 \\[4pt]
& \leq \|\S(y^{k+1}-y^k)+\sigma \A \B^*(z^{k+1}-z^k)+(1-\tau^{-1}) \A(x^{k+1}-x^k)\|^2\\[4pt]
&+\|\T(z^{k+1}-z^k)+(1-\tau^{-1})\B(x^{k+1}-x^k)\|^2+{ {(\tau \sigma )^{-2}}}\|x^{k+1}-x^k\|^2\\[4pt]
&\leq \left[3\|\S \|\|y^{k+1}-y^k\|_{\S}^2+3\sigma ^2\lambda_{\max}(\A \A^*) \|\B(z^{k+1}-z^k)\|^2+3(1-\tau^{-1})^2\|\A(x^{k+1}-x^k)\|^2\right]\\[4pt]
&\quad +\left[2\|\T\|\|z^{k+1}-z^k\|_\T^2+2(1-\tau^{-1})^2\|\B(x^{k+1}-x^k)\|^2+{ {(\tau \sigma )^{-2}}}\|x^{k+1}-x^k\|^2\right]\\[4pt]
&\leq \kappa_1 \|y^{k+1}-y^k\|_{\S}^2+\kappa_2\|z^{k+1}-z^k\|^2_{\T+\sigma \B \B^*}+\kappa_3 (\tau^2\sigma )^{-1}\|x^{k+1}-x^k\|^2,
\end{array}
$$
which immediately implies \eqref{cou}.
\end{proof}

\medskip

For any $\tau\in (0, \infty)$, define
$$
s_\tau:=\displaystyle \frac{5-\tau-3\min \{\tau,\tau^{-1}\}}{4} \quad \& \quad  t_\tau:=\displaystyle \frac{1-\tau+\min \{\tau,\tau^{-1}\}}{8}.
$$
Note that
\[\label{eq:t-stau}1/4\le s_\tau\le (5-2\sqrt{3})/4 \quad \& \quad 0< t_\tau\le 1/8 , \quad \forall\, \tau \in ( 0, (1+\sqrt{5})/2).
\]
Denote
\[\label{M1}
\M:={\rm Diag}\left(
\S+\Sigma_{\vartheta_g}, \T+\Sigma_{\varphi_h}+\sigma \B\B^*,  (\tau\sigma )^{-1}{\cal I} \right)+s_\tau \sigma  \E\E^*
\]
and
\[
 \H: ={\rm Diag}\left(\S+\displaystyle \frac{1}{2}\Sigma_{\vartheta_g} , \T+\displaystyle \frac{1}{2}\Sigma_{\varphi_h}+\displaystyle \tau \sigma\B\B^*, 4t_\tau(\tau^2\sigma)^{-1}{\cal I} \right)+t_\tau\sigma \E\E^*. \]
Then we immediately get the following relation
\[\label{importin1}
\kappa \H\succeq \min \{\tau, 4t_\tau\}   \H_0 +\kappa  t_\tau \sigma  \E\E^*, \quad \forall\, \tau \in ( 0, (1+\sqrt{5})/2)
.\]

\begin{proposition}\label{lem:pd} Let $\tau \in ( 0, (1+\sqrt{5})/2)$. Then
$$
  \Sigma_{\vartheta_g} +\S+ \sigma  \A\A^*\succ 0 \ \& \ \Sigma_{\varphi_h} +\T+ \sigma \B\B^*\succ 0\Longleftrightarrow
  \M\succ 0 \Longleftrightarrow   \H\succ 0.
$$
\end{proposition}
\begin{proof} Since, in view of  $(\ref{eq:t-stau})$, it is obvious that $\M\succ 0 \Longleftrightarrow   \H\succ 0$, we only need to show that
$$
  \Sigma_{\vartheta_g} +\S+ \sigma  \A\A^*\succ 0 \ \ \& \ \ \Sigma_{\varphi_h} +\T+ \sigma \B\B^*\succ 0\Longleftrightarrow
  \M\succ 0.
$$
First, we show that $\Sigma_{\vartheta_g} +\S+ \sigma  \A\A^*\succ 0 \ \& \ \Sigma_{\varphi_h} +\T+ \sigma \B\B^*\succ 0\Longrightarrow
  \M\succ 0.$  Suppose that $\Sigma_{\vartheta_g} +\S+ \sigma  \A\A^*\succ 0 \ \& \ \Sigma_{\varphi_h} +\T+ \sigma \B\B^*\succ 0$, but
  there exists a vector $0\ne d:=(d_y, d_z,d_x) \in \Y\times \Z\times \X$ such that $\langle d, \M d\rangle =0$. By using the definition of $\M$ and (\ref{eq:t-stau}), we have
  $$
  d_x=0, \ (\Sigma_{\varphi_h} +\T+ \sigma \B\B^*)d_z=0, \ (\Sigma_{\vartheta_g} +\S)y=0 \ \& \ \E^*(d_y,d_z,0)=0,
  $$
   which, together with the assumption that   $\Sigma_{\vartheta_g} +\S+ \sigma  \A\A^*\succ 0 \ \& \ \Sigma_{\varphi_h} +\T+ \sigma \B\B^*\succ 0$, imply $d=0$. This contradiction shows that $\M\succ 0$.

   Next, suppose that $\M\succ 0$. Since $s_\tau>0$ and for any $d=(0,d_z,0) \in \Y\times \Z\times \X$, $\langle d, \M d \rangle =\left \langle d_z,
  ( \Sigma_{\varphi_h} +\T+ (1 +s_\tau )\sigma\B\B^* )d_z\right \rangle $, we know that $\Sigma_{\varphi_h} +\T+  \sigma\B\B^*\succ 0$. Similarly,
  since for any $d=(d_y,0,0) \in \Y\times \Z\times \X$, $\langle d, \M d \rangle =\left \langle d_y,
  ( \Sigma_{\varphi_g} +\S+   s_\tau  \sigma\A\A^* )d_y\right \rangle $, we know that  $\Sigma_{\vartheta_g} +\S+ \sigma  \A\A^*\succ 0$. So the proof is completed.
\end{proof}

\begin{proposition}\label{Han-T-11}
Let $\tau \in (0,(1+\sqrt{5})/2)$ and $\{(y^k,z^k,x^k)\}$ be an infinite sequence generated by the sPADMM.  Then for any $\bar u=(\bar y,\bar z,\bar x) \in \overline{\Omega}$ and any $k\ge 1$,
\[\label{les11}
\|u^{k+1}-\bar u\|^2_{\M}+\|z^{k+1}-z^k\|^2_\T\leq  \left( \|u^{k}-\bar u\|^2_{\M}+\|z^{k}-z^{k-1}\|^2_\T \right)-\|u^{k+1}-u^k\|^2_{\H}.
\]
Consequently, we have   for all $k\ge 1$,
\[\label{les11-dist}
{\rm dist}^2_{\M} (u^{k+1},{\overline \Omega} )+\|z^{k+1}-z^k\|^2_\T\leq
 \left( {\rm dist}^2_{\M} (u^{k},{\overline \Omega} )+\|z^{k}-z^{k-1}\|^2_\T\right)-  \|u^{k+1}-u^k\|^2_{\H}.
\]
\end{proposition}
\begin{proof} Let $\bar u=(\bar y,\bar z,\bar x) \in \overline{\Omega}$ be fixed but arbitrarily chosen.
From part  (i)   of Theorem \ref{Th:B1}, we have  for
  $k\ge 1$ that
\[\label{111a}
\begin{array}{l}
(\tau\sigma )^{-1}\|x^{k+1}-{\bar x}\|^2+\|y^{k+1}-{\bar y}\|^2_\S+\|z^{k+1}-{\bar z}\|^2_\T+\sigma \|\B^*(z^{k+1}-{\bar z})\|^2\\
\quad \quad +\|z^{k+1}-z^k\|^2_\T+\left(1-\min\{\tau,\tau^{-1}\}\right)\sigma \|\E^*(y^{k+1},z^{k+1},0)-c\|^2\\
\leq (\tau\sigma )^{-1}\|x^k-{\bar x}\|^2+\|y^k-{\bar y}\|^2_\S+\|z^k-{\bar z}\|^2_\T+\sigma \|\B^*(z^k-{\bar z})\|^2\\
\quad \quad +\|z^k-z^{k-1}\|^2_\T+\left(1-\min\{\tau,\tau^{-1}\}\right)\sigma \|\E^*(y^{k},z^k,0)-c\|^2\\
-\left\{\sigma [\tau-\tau^2+\tau \min\{\tau,\tau^{-1}\}] \|\B^*(z^{k+1}-z^k)\|^2\right.\\
+\|z^{k+1}-z^k\|^2_\T+\|y^{k+1}-y^k\|^2_\S+2\|y^{k+1}-{\bar y}\|^2_{\Sigma_{\vartheta_g}}+2\|z^{k+1}-{\bar z}\|^2_{\Sigma_{\varphi_h}}\\
\left.+(1-\tau+\min \{\tau,\tau^{-1}\})\sigma  \|\E^*(y^{k+1},z^{k+1},0)-c\|^2\right\}.
\end{array}
\]
By reorganizing the terms in (\ref{111a}), we obtain
$$\begin{array}{l}
(\tau\sigma )^{-1}\|x^{k+1}-{\bar x}\|^2+\|y^{k+1}-{\bar y}\|^2_\S+\|z^{k+1}-{\bar z}\|^2_\T+\sigma  \|\B^*(z^{k+1}-{\bar z},0)\|^2\\
\quad \quad
+\|z^{k+1}-z^k\|^2_\T+\frac{1}{4}\left(5-\tau-3\min \{\tau,\tau^{-1}\}\right)\sigma \|\E^*(y^{k+1},z^{k+1},0)-c\|^2\\
\quad \quad
+\|y^{k+1}-{\bar y}\|^2_{\Sigma_\vartheta}
+\|z^{k+1}-{\bar z}\|^2_{\Sigma_\varphi}\\
\leq (\tau\sigma )^{-1}\|x^k-{\bar x}\|^2+\|y^k-{\bar y}\|^2_\S+\|z^k-{\bar z}\|^2_\T+\sigma \|\B^*(z^k-{\bar z})\|^2\\
\quad \quad +\|z^k-z^{k-1}\|^2_\T+\frac{1}{4}\left(5-\tau-3\min \{\tau,\tau^{-1}\}\right)\sigma  \|\E^*(y^{k},z^k,0)-c\|^2\\
\quad \quad +\|y^{k}-{\bar y}\|^2_{\Sigma_\vartheta}
+\|z^{k}-\bar z\|^2_{\Sigma_\varphi}\\
\,-\left\{ \sigma  \tau \|\B^*(z^{k+1}-z^k)\|^2+\|z^{k+1}-z^k\|^2_\T+\|y^{k+1}-y^k\|^2_\S\right.\\
\quad \quad +\|y^{k+1}-{\bar y}\|^2_{\Sigma_\vartheta}+\|y^{k}-{\bar y}\|^2_{\Sigma_\vartheta}
+\|z^{k+1}-{\bar z}\|^2_{\Sigma_\varphi}+\|z^{k}-{\bar z}\|^2_{\Sigma_\varphi}\\
\quad +\frac{1}{2}(1-\tau+\min \{\tau,\tau^{-1}\})\sigma  \|\E^*(y^{k+1},z^{k+1},0)-c\|^2\\
\quad \left.+\frac{1}{4}(1-\tau+\min \{\tau,\tau^{-1}\})\sigma [ \|\E^*(y^{k+1},z^{k+1},0)-c\|^2+
 \|\E^*(y^{k},z^{k},0)-c\|^2]\right\}
\end{array}
$$
or equivalently
\[\label{112aa}
\begin{array}{l}
(\tau\sigma )^{-1}\|x^{k+1}-\bar x\|^2+\|y^{k+1}-\bar y\|^2_\S+\|z^{k+1}-\bar z\|^2_\T+\sigma  \|\B^*(z^{k+1}-\bar {z})\|^2\\
\quad \quad
+\|z^{k+1}-z^k\|^2_\T+s_\tau\sigma \|\E^*(y^{k+1},z^{k+1},0)-c\|^2\\
\quad \quad
+\|y^{k+1}-\bar y\|^2_{\Sigma_\vartheta}
+\|z^{k+1}-\bar z\|^2_{\Sigma_\varphi}\\
\leq (\tau\sigma )^{-1}\|x^k-\bar x\|^2+\|y^k-\bar y\|^2_\S+\|z^k-\bar z\|^2_\T+\sigma \|\B^*(z^k-{\bar z})\|^2\\
\quad \quad +\|z^k-z^{k-1}\|^2_\T+s_\tau\sigma  \|\E^*(y^{k},z^k,0)-c\|^2+\|y^{k}-{\bar y}\|^2_{\Sigma_\vartheta}
+\|z^{k}-{\bar z}\|^2_{\Sigma_\varphi}\\
-\left\{ \sigma \tau \|\B^*(z^{k+1}-z^k)\|^2+\|z^{k+1}-z^k\|^2_\T+\|y^{k+1}-y^k\|^2_\S \right.\\
\quad \quad +\|y^{k+1}-{\bar y}\|^2_{\Sigma_\vartheta}+\|y^{k}-{\bar y}\|^2_{\Sigma_\vartheta}
+\|z^{k+1}-{\bar z}\|^2_{\Sigma_\varphi}+\|z^{k}-{\bar z}\|^2_{\Sigma_\varphi}\\
\quad + \frac{1}{2}\left(1-\tau+\min \{\tau,\tau^{-1}\}\right)\sigma \|\E^*(y^{k+1},z^{k+1},0)-c\|^2\\
\quad +  \frac{1}{4}(1-\tau+\min \{\tau,\tau^{-1}\})\sigma [ \|\E^*(y^{k+1},z^{k+1},0)-c\|^2+
\left. \|\E^*(y^{k},z^{k},0)-c\|^2]\right\}.
\end{array}
\]
Using equalities
$$\begin{array}{l}
\E^*(y^{k+1},z^{k+1},0)-c=\A^*(y^{k+1}-\bar y)+\B^*(z^{k+1}-\bar  z),\\
\E^*(y^k,z^k,0)-c=\A^*(y^k-\bar y)+\B^*(z^k-\bar z),\\
\E^*(y^{k+1},z^{k+1},0)-c=(\tau\sigma )^{-1}(x^{k+1}-x^{k})
\end{array}
$$
 and inequalities
 $$\begin{array}{l}
 \|y^{k+1}-\bar  y\|^2_{\Sigma_{\vartheta_g}}+\|y^{k}-\bar y\|^2_{\Sigma_{\vartheta_g}}\geq  \frac{1}{2}\|y^{k+1}-y^k\|^2_{\Sigma_{\vartheta_g}},\\
\|z^{k+1}-\bar  z\|^2_{\Sigma_{\varphi_h}}+\|z^{k}-\bar  z\|^2_{\Sigma_{\varphi_h}}\geq  \frac{1}{2}\|z^{k+1}-z^k\|^2_{\Sigma_{\varphi_h}},\\
 \|\E^*(y^{k+1},z^{k+1},0)-c\|^2+
\|\E^*(y^{k},z^{k},0)-c\|^2\geq  \frac{1}{2}  \|\A^*(y^{k+1}-y^k)+\B^*(z^{k+1}-z^k)\|^2,
 \end{array}
 $$
 we obtain from (\ref{112aa}) and the definitions of  $s_\tau$ and $t_\tau$ that
$$\begin{array}[b]{l}
(\tau\sigma )^{-1}\|x^{k+1}-{\bar x}\|^2+\|y^{k+1}-\bar y\|^2_\S+\|z^{k+1}-{\bar z}\|^2_\T+\sigma \|\B^*(z^{k+1}-{\bar z})\|^2\\
\quad \quad
+\|z^{k+1}-z^k\|^2_\T+s_\tau\sigma  \|\A^*(y^{k+1}-\bar  y)+\B^*(z^{k+1}-\bar z)\|^2\\
\quad \quad
+\|y^{k+1}-\bar y\|^2_{\Sigma_{\vartheta_g}}
+\|z^{k+1}-\bar z\|^2_{\Sigma_{\varphi_h}}\\
\leq (\tau\sigma )^{-1}\|x^k-\bar x\|^2+\|y^k-\bar y\|^2_\S+\|z^k-\bar z\|^2_\T+\sigma \|\B^*(z^k-\bar z)\|^2\\
\quad \quad +\|z^k-z^{k-1}\|^2_\T+s_\tau\sigma  \|\A^*(y^{k}-\bar y)+\B^*(z^k-\bar z)\|^2+\|y^{k}-\bar y\|^2_{\Sigma_{\vartheta_g}}
+\|z^{k}-\bar z\|^2_{\Sigma_{\varphi_h}}\\
-\big\{ \sigma \tau \|\B^*(z^{k+1}-z^k)\|^2+\|z^{k+1}-z^k\|^2_\T+\|y^{k+1}-y^k\|^2_\S
 + \frac{1}{2}\|y^{k+1}-y^{k}\|^2_{\Sigma_{\vartheta_g}}\\
\quad + \frac{1}{2}\|z^{k+1}-z^{k}\|^2_{\Sigma_{\varphi_h}}+4t_\tau(\tau^2\sigma)^{-1}\|x^{k+1}-x^k\|^2\\
\quad +t_\tau\sigma \|\A^*(y^{k+1}-y^k)+\B^*(z^{k+1}-z^k)\|^2\big\},
\end{array}
$$
which shows that  (\ref{les11}) holds.  By noting that $\overline{\Omega}$ is a nonempty closed convex set and (\ref{les11}) holds for
any $\bar{u} \in \overline{\Omega}$, we immediately get (\ref{les11-dist}).
\end{proof}

For establishing the linear  rate of convergence of the sPADMM, we need the following error bound condition.

\begin{ass}[{Error bound condition}]\label{erB2} For any given $\bar{u} \in \overline{\Omega}$,
there exist positive constants $\delta  $  and  $\eta> 0$ such that
\begin{eqnarray}\label{oi1a}
{\rm dist}(u,\overline{\Omega})\leq\eta\|R(u)\|,\quad\forall\, u\in \{ u\in \U: \, \|u-\bar{u}\|\le \delta\}.
\end{eqnarray}
\end{ass}

\begin{theorem}\label{th-1-rate}
Let $\tau \in (0,(1+\sqrt{5})/2)$.  Suppose that Assumptions \ref{ass.1} and \ref{erB2} hold. Assume also that both $ \Sigma_{\vartheta_g} +\S+ \sigma  \A\A^*$ and $ \Sigma_{\varphi_h} +\T+ \sigma \B\B^*$ are positive definite. Let $\{(y^k,z^k,x^k)\}$ be the infinite sequence generated from the sPADMM. Then for all $k$ sufficiently large,
\[\label{rate-tuales11}
{\rm dist}^2_{\M} (u^{k+1},{\overline \Omega} )+\|z^{k+1}-z^k\|^2_\T \leq \mu \left[{\rm dist}^2_{\M} (u^{k},{\overline \Omega})+\|z^{k}-z^{k-1}\|^2_\T\right],
\]
where
$$\mu:=\displaystyle (1+2 \kappa_4)^{-1}(1+\kappa_4)<1 \quad \&\quad   \kappa_4:=\displaystyle \min \{\tau, 4t_\tau\}\left(\eta^2\kappa \lambda_{\max}(\M)\right)^{-1}>0.
$$
Moreover, there exists a positive number  $\varsigma \in [\mu, 1)$ such that  for all $k\ge 1$,
\[\label{rate-tuales11-a}
{\rm dist}^2_{\M} (u^{k+1},{\overline \Omega} )+\|z^{k+1}-z^k\|^2_\T \leq \varsigma \left[{\rm dist}^2_{\M} (u^{k},{\overline \Omega})+\|z^{k}-z^{k-1}\|^2_\T\right].
\]
 \end{theorem}
\begin{proof}
From Theorem \ref{Th:B1} we know that   the whole sequence $\{(y^k,z^k,x^k)\}$ generated by the sPADMM  converges to     a KKT point in $\overline{\Omega}$, say $\bar {u} =(\bar{y}, \bar{z}, {\bar x})$.
Combining  Assumption \ref{erB2} with Lemma \ref{jiu} we know that there exists a constant  $\eta>0$  that for all $k$ sufficiently large,
\begin{eqnarray}\label{D1}
{\rm dist}^2(u^{k+1},\overline{\Omega})\leq\eta^2\|R(u^{k+1})\|^2\le \eta^2\|u^k-u^{k+1}\|_{  \H_0}^2.
\end{eqnarray}

From the definition of $\H$, we have  for all $k\ge 0$,
$$\|z^{k+1}-z^k\|_\T^2\leq \|u^{k+1}-u^k\|^2_\H.
$$
It follows from (\ref{importin1}) and   (\ref{D1})  that for all $k$ sufficiently large,  
\[\label{eq:1021}
\begin{array}{ll}
\|u^{k+1}-u^k\|^2_{\H} & \geq \min \{\tau, 4t_\tau\}\kappa^{-1}\|u^{k+1}-u^k\|^2_{  \H_0}\\[4pt]
& \geq \displaystyle \min \{\tau, 4t_\tau\}\kappa^{-1}\eta^{-2} {\rm dist}^2(u^{k+1},{\overline \Omega})\geq  \kappa _4{\rm dist}_{\M}^2(u^{k+1},{\overline \Omega}).
\end{array}
\]
Let $\kappa_5=(1+ \kappa_4)^{-1}$. From (\ref{les11-dist}) in Proposition \ref{Han-T-11} and (\ref{eq:1021}), we have for all $k$ sufficiently large that
\[\label{eq:1031}
\begin{array}{l}
{\rm dist}^2_{\M} (u^{k+1},{\overline \Omega} )+\|z^{k+1}-z^k\|^2_\T-\left \{ {\rm dist}^2_{\M} (u^{k},{\overline \Omega})+\|z^{k}-z^{k-1}\|^2_\T\right\}\\
\leq -\left ((1-\kappa_5)\|u^{k+1}-u^k\|^2_{\H}+ \kappa_5 \|u^{k+1}-u^k\|^2_{\H}\right)\\
\leq -\left((1- \kappa_5) \|z^{k+1}-z^k\|^2_\T+{\kappa_5}  {\kappa_4} {\rm dist}_{\M}^2(u^{k+1},{\overline \Omega})\right).
\end{array}
\]
 Then we obtain from
(\ref{eq:1031}) that for all $k$ sufficiently large,
$$
(1+{\kappa_5}{ \kappa_4}) {\rm dist}_{\M}^2(u^{k+1},{\overline \Omega})+(2-{\kappa_5})\|z^{k+1}-z^k\|^2_\T\leq {\rm dist}^2_{\M} (u^{k},{\overline \Omega})+\|z^{k}-z^{k-1}\|^2_\T.
$$
By noting that $1+{\kappa_5}  \kappa_4=2-{\kappa_5}=\mu^{-1}$, we obtain the estimate (\ref{rate-tuales11}).

By combining (\ref{rate-tuales11}) with   Lemma \ref{jiu},  (\ref{importin1}) and (\ref{les11-dist}) in Proposition  \ref{Han-T-11}, we can obtain directly that there exists  a positive number  $\varsigma \in [\mu, 1)$ such that  (\ref{rate-tuales11-a}) holds for all $k\ge 1$. The proof is completed.
\end{proof}

Theorem \ref{th-1-rate} provides a very general result on the  linear rate of convergence for the sPADMM under a fairly mild error bound assumption, which holds automatically if $R^{-1}$ is  piecewise polyhedral.  Since $R^{-1}$ is  piecewise polyhedral if and only if $R$ itself  is piecewise polyhedral, we obtain the following directly from Theorem  \ref{th-1-rate}, Proposition \ref{prop:piecewisePoly} and Lemma \ref{jiu}.

\begin{corollary}\label{cor-1-rate}
Let $\tau \in (0,(1+\sqrt{5})/2)$.  Suppose that $\overline{\Omega} \ne \emptyset$    and that  both $ \Sigma_{\vartheta_g} +\S+ \sigma  \A\A^*$ and $ \Sigma_{\varphi_h} +\T+ \sigma \B\B^*$ are positive definite. Assume that  the mapping  $R: {\cal U} \to {\cal U}$ is piecewise polyhedral. Then  there exists a constant $\varsigma \in (0,1)$ such that  the infinite sequence  $\{(y^k,z^k,x^k)\}$    generated from the sPADMM satisfies
\[\label{rate-tuales11-poly}
{\rm dist}^2_{\M} (u^{k+1},{\overline \Omega} )+\|z^{k+1}-z^k\|^2_\T \leq \varsigma \left[{\rm dist}^2_{\M} (u^{k},{\overline \Omega})+\|z^{k}-z^{k-1}\|^2_\T\right], \quad \forall\, k\ge 1.
\]
\end{corollary}

\subsection{Applications to convex composite quadratic conic programming}\label{ccqcp}
 In this subsection  we shall demonstrate how the just established  linear rate convergence theorem  can be applied to the following convex composite   quadratic conic programming
\begin{equation}\label{eq:conicP0}
\begin{array}{ll}
\min & \displaystyle \frac{1}{2}\langle x, {\cal Q} x\rangle+\langle c,x\rangle+\phi(x)\\
{\rm s.t.} & {\cal A}x =b,\ \   x\in {\cal K},
\end{array}
\end{equation}
where $c \in \X$,   $b \in \Re^m$, ${\mathcal Q}: \X\to \X$ is a self-adjoint positive semi-definite linear operator,  ${\cal A}: \X \to \Re^m$ is a linear operator, ${\cal K} $ is a closed convex cone in $\X$ and $\phi: \X \in (-\infty, \infty]$ is a  proper closed convex function whose epigraph is convex polyhedral, i.e., $\phi$ is a closed proper convex polyhedral function.  If ${\cal K} $ is a polyhedral cone, problem (\ref{eq:conicP0}) is called the convex  composite   quadratic programming (QP).

 By introducing an additional variable $u\in \X$, we can rewrite problem (\ref{eq:conicP0})  equivalently as
\begin{equation}\label{eq:conicP}
\begin{array}{ll}
\min & \displaystyle \frac{1}{2}\langle x, {\cal Q} x\rangle+\langle c,x\rangle +\delta_{{\cal K}}(x)+\phi(u)\\
{\rm s.t.} & {\cal A}x =b,\  \   x-u=0.
\end{array}
\end{equation}
Obviously, problem (\ref{eq:conicP}) is in the form of (\ref{ad}). Let the polar of ${\cal K}$ be defined by ${\cal K}^\circ: =\{d\in \X: \langle d, x\rangle \le 0,  \ \forall\, x\in {\cal K}\}$. Denote  the dual cone  of ${\cal K}$ by ${\cal K}^* :=-{\cal K}^\circ$.
The Lagrange dual of problem (\ref{eq:conicP}) takes the form of
$$\begin{array}{ll}
\max &\displaystyle \inf_{x\in \X} \left\{\frac{1}{2}\langle x, {\cal Q} x\rangle +\langle v , x\rangle \right\}+\langle b, y\rangle-\phi^* (-z)\\
{\rm s.t.} & s+ {\cal A}^*y+v+z=c,\ \  s\in {\cal K}^* ,
\end{array}
$$%
which is equivalent to
\begin{equation}\label{eq:conicD}
\begin{array}{ll}
\min & \displaystyle\delta_{{\cal K}^*}(s)-\langle b, y\rangle+\frac{1}{2}\langle w, {\cal Q} w\rangle  +\phi^*(-z)\\
{\rm s.t.} & s+{\cal A}^*y- {\cal Q}w+z=c, \ \  w \in {\cal W},
\end{array}
\end{equation}
where ${\cal W}$ is any linear subspace in $\X$   containing
${\rm Range\,} {\cal Q}$, the range space of ${\cal Q}$, e.g., ${\cal W}=\X$ or ${\cal W}={\rm Range\,} {\cal Q}$. When ${\cal W}=\X$,    problem (\ref{eq:conicD}) is better known as the Wolfe dual to  problem (\ref{eq:conicP}) (see Fujiwara, Han and  Mangasarian \cite{FHMangasarian84} for discussions on the Wolfe dual of conventional nonlinear programming   and Qi \cite{Qi09} on nonlinear semi-definite  programming). So when ${\rm Range\,} {\cal Q}\subseteq {\cal W}\ne \X$, one may call problem (\ref{eq:conicD})
the restricted Wolfe dual to  problem (\ref{eq:conicP}). One particularly useful  case is the restricted Wolfe dual with ${\cal W}={\rm Range\,} {\cal Q}$. The dual problem (\ref{eq:conicD}) has four natural variable-blocks and can be written in the form of (\ref{ad}) in several different ways. The cases that we are interested in applying the sPADMM are: 1)  if ${\cal K} \ne \X$, then $(s,y,w)$ is treated as one variable-block and $z$ the other block; and   2) if ${\cal K} =\X$, then $(w,y)$ is treated as one variable-block and $s$ the other block. We shall only discuss  case 1) as case 2) can be done similarly  in a simpler manner.

First, we consider the application of the sPADMM to the primal problem (\ref{eq:conicP}).  The augmented Lagrangian function ${\cal L}_{\sigma}^P$ for problem (\ref{eq:conicP}) is defined as follows
$$\begin{array}{lcl}
{\cal L}_{\sigma}^P(x,u;y,z)&:=&\displaystyle \frac{1}{2}\langle x, {\cal Q} x\rangle+\langle c,x\rangle +\delta_{{\cal K}}(x)+\phi(u)+\langle y,b-\A x\rangle+\langle z, u-x\rangle
\\[4pt]
& & +
\displaystyle \frac{\sigma}{2}(\|b-\A x\|^2 +\|u-x\|^2), \quad \forall\, (x,u,y,z)\in \X\times \X\times \Re^m \times \X.
\end{array}
$$

\medskip
\fbox{
 \begin{minipage}[t]{0.9\textwidth}{\noindent\bf sPADMM}: A semi-proximal alternating direction method of multipliers for solving the convex optimization problem (\ref{eq:conicP}).

\begin{itemize}
\item[ ]\hskip -.5cm{\bf Step 0.} Input $(x^0, u^0,y^0,z^0)\in {\cal K} \times {\rm dom\, }(\phi)\times \Re^m \times \X.$ Let  $\tau\in (0,\infty)$ be a positive parameter (e.g., $\tau \in ( 0, (1+\sqrt{5})/2)$\,).
   Define $\S: \X\to \X$ to be any self-adjoint  positive semi-definite linear operator, e.g.,  $\S:=0$ if   ${\cal K} =\X$ and
    $\S:= \lambda_{\max}\left ( {\cal Q}+\sigma \A^* \A\right) {\cal I} -\left( {\cal Q}+\sigma \A^* \A\right)
    $ if  $ {\cal K} \ne \X $.
   Set $k:=0$.

\item[ ]\hskip -.5cm{\bf Step 1.} Set
 \begin{equation*}
\left\{\begin{aligned}
&x^{k+1}= \hbox{arg}\min \,{\cal L}_{\sigma}^P(x,u^k;y^k,z^k) +\frac{1}{2}\|x-x^k\|^2_\S\, ,\\
&u^{k+1}= \hbox{arg}\min\, {\cal L}_{\sigma}^P(x^{k+1},u;y^{k},z^k)\, ,\\
&y^{k+1}=y^k+\tau\sigma (b-\A x^{k+1}) \ \ \& \ \ z^{k+1} = z^k + \tau\sigma (u^{k+1}-x^{k+1}).
\end{aligned}\right.
\end{equation*}
\item[ ]\hskip -.5cm{\bf Step 2.} If a termination criterion is not met, set $k:=k+1$ and go to Step 1.
\end{itemize}
 \end{minipage} }

\medskip
It is easy to see from Theorem \ref{th-1-rate} that as long as Assumptions \ref{ass.1} and \ref{erB2}  for problem (\ref{eq:conicP}) hold and $\tau \in ( 0, (1+\sqrt{5})/2)$, the infinite sequence $\{(x^k,u^k,y^k,z^k)\}$ generated by the sPADMM for solving   problem (\ref{eq:conicP}) converges to a KKT point of problem (\ref{eq:conicP}) globally at a   linear rate. Note that Assumption  \ref{erB2} holds automatically if ${\cal K}$ is   convex polyhedral,  e.g., ${\cal K} =\X$ or ${\cal K} =\Re^n_+$.

Next, we turn to the dual problem (\ref{eq:conicD}). As mentioned earlier,   problem (\ref{eq:conicD}) has four natural variable-blocks. Since the directly extended ADMM to the multi-block case may be divergent even the dual setp-length $\tau $ is taken to be as small as $10^{-8}$ \cite{CHYYuan2014}, one needs new ideas   to deal with  problem (\ref{eq:conicD}). Here, we will adopt the smart symmetric Gauss-Seidel (sGS) technique invented by
Li et al. \cite{LSToh2014}. For details on the sGS technique, see \cite{LiXDThesis2015}.
 Most recent research  has shown that it is much more efficient to solve the dual problem (\ref{eq:conicD}) rather than its primal counterpart (\ref{eq:conicP}) in the context of semi-definite programming and convex quadratic semi-definite programming \cite{STYang2015,LSToh2014,LiXDThesis2015,CSToh2015}. At the first glance, this seems to be counter-intuitive as problem (\ref{eq:conicD}) looks much more complicated than the primal problem (\ref{eq:conicP}). The key point for the more efficiency in dealing with the dual problem is to intelligently combine  the above mentioned sGS technique with the sPADMM, which will be shown below.

 The augmented Lagrangian function ${\cal L}_{\sigma}^D$ for problem (\ref{eq:conicD}) is defined as follows
$$\begin{array}{lcl}
{\cal L}_{\sigma}^D(s,y,w,z;x)&:=&\displaystyle \delta_{{\cal K}^*}(s)-\langle b, y\rangle+\frac{1}{2}\langle w, {\cal Q} w\rangle  +\phi^*(-z)+\langle x,s+\A^*y -{\cal Q}w + z - c\rangle
\\[4pt]
& & +
\displaystyle \frac{\sigma}{2}\| s+\A^*y -{\cal Q}w + z - c\|^2, \ \forall\, (s,y,w,z,x)\in \X\times\Re^m\times \W \times \X \times \X.
\end{array}
$$

\medskip
\fbox{
 \begin{minipage}[t]{0.9\textwidth}{\noindent\bf sGS-sPADMM}: A symmetric Gauss-Seidel based  semi-proximal alternating direction method of multipliers for solving  problem (\ref{eq:conicD}).

\begin{itemize}
\item[ ]\hskip -.5cm{\bf Step 0.} Input $(s^0,y^0,w^0,z^0,x^0)\in{\cal K}^*\times\Re^m\times{\cal W}\times(-\hbox{dom}\; \phi^*)\times \mathcal{X}.$ Let  $\tau\in (0,\infty)$ be a positive parameter (e.g., $\tau \in ( 0, (1+\sqrt{5})/2)$\,).   Choose any two  self-adjoint  positive semi-definite linear operators $\S_1: \Re^m \to \Re^m$ and $\S_2: \W \to \W$ satisfying $\S_1 +\sigma \A\A^* \succ 0$ and $\S_2 +{\cal Q} +\sigma {\cal Q}^2\succ 0$.
       Set $k:=0$.
\item[ ]\hskip -.5cm{\bf Step 1.} Set
\begin{equation*}
\left\{
\begin{aligned}
&{w}^{k+\frac{1}{2}} = \hbox{arg}\min\, \L_{\sigma}^D(s^k,y^k,w,z^k;x^k) +\displaystyle \frac{1}{2}\|w-w^k \|^2_{{\cal S}_2} \, ,\\
&{ y}^{k+\frac{1}{2}} = \hbox{arg}\min\, \L_{\sigma}^D(s^k,y,w^{k+\frac{1}{2}},z^k;x^k)+ \frac{1}{2}\|y-y^k \|^2_{{\cal S}_1}\, ,\\
&s^{k+1} = \hbox{arg}\min\, \L_{\sigma}^D(s,y^{k+\frac{1}{2}},w^{k+\frac{1}{2}},z^k;x^k)\, ,\\
&y^{k+1} = \hbox{arg}\min\, \L_{\sigma}^D(s^{k+1},y,w^{k+\frac{1}{2}},z^k;x^k)+ \frac{1}{2}\|y-y^k \|^2_{{\cal S}_1}\, ,\\
&w^{k+1} = \hbox{arg}\min\, \L_{\sigma}^D(s^{k+1},y^{k+1},w,z^k;x^k) +\displaystyle \frac{1}{2}\|w-w^k \|^2_{{\cal S}_2} \, ,\\
&z^{k+1} = \hbox{arg}\min\, \L_{\sigma}^D(s^{k+1},y^{k+1},w^{k+1},z;x^k)\, ,\\
&x^{k+1}=x^k+\tau\sigma (s^{k+1}+\A^*y^{k+1} -{\cal Q}w^{k+1} + z^{k+1} - c).
\end{aligned}
\right.
\end{equation*}
\item[ ]\hskip -.5cm{\bf Step 2.} If a termination criterion is not met, set $k:=k+1$ and go to Step 1.
\end{itemize}
 \end{minipage} }

 \medskip

Note that in the above Algorithm sGS-sPADMM, one can always choose ${\cal S}_1 = 0$ if $\A:\X\to\Re^m$ is surjective and ${\cal S}_2=0$ if ${\cal W} = {\rm Range}\, ({\cal Q})$.
The global convergence of Algorithm sGS-sPADMM is established in \cite{LSToh2014} by connecting it into an equivalent sPADMM scheme \eqref{xna}--\eqref{xnc} for solving a particular  problem of the form \eqref{ad}.
By using the same connection, just as for the primal case, one can use Theorem \ref{th-1-rate} to derive the linear rate convergence of the infinite sequence $\{(s^k,y^k,w^k,z^k,x^k)\}$ generated by Algorithm sGS-sPADMM  if Assumptions \ref{ass.1} and \ref{erB2}  hold for problem (\ref{eq:conicD}) and $\tau \in ( 0, (1+\sqrt{5})/2)$. As mentioned earlier, Assumption  \ref{erB2} holds automatically if ${\cal K}$ is   convex polyhedral. However, for a non-polyhedral ${\cal K}$, there exist few results about the existence of the error bound condition as in Assumption \ref{erB2} except for ${\cal K}$ to be either a second order cone \cite{BR2005} or an SDP cone \cite{Sun2006}, where the strong regularity introduced by Robinson \cite{Robinson1979} is characterised in terms of the strong second order sufficient condition and the constraint nondegeneracy. The strong regularity provides a sufficient condition for Assumption \ref{erB2} to hold. Since  the isolated calmness condition given in Definition \ref{def:calm} is a much weaker condition than the strong regularity,  in the next two sections, we shall conduct a thorough study on the isolated calmness in the context of composite semi-definite, convex and non-convex, optimization problems. The obtained results on the isolated calmness are not only useful for deriving the linear rate convergence of the sPADMM but also represent substantial advancements in the context of second order variational analysis for conic optimization problems constrained with non-polyhedral  convex cones. As a final note to this section, we comment  that in all the above applications, while the linear operator $\S$ may take various values,  the linear operator  $\T\equiv 0$.

\section{Characterizations of the isolated calmness for  semi-definite optimization problems}\label{section:isolatedcalmness}
 \setcounter{equation}{0}
Let  ${\cal Z}$ be a  finite dimensional real Euclidean  space. For an integer $p >0$, let  $\mathbb{S}^{p}_+$ be the positive semi-definite cone of all symmetric positive semi-definite matrices in  the space $\mathbb{S}^{p}$ of $p$ by $p$ real symmetric matrices. Denote ${\cal Y}:=\mathbb{S}^{p} \times {\cal Z}$ and  $\mathbb{S}^{p}_-:=-\mathbb{S}^{p}_+=(\mathbb{S}^{p}_+)^\circ$.   Next, we shall consider the isolated calmness for the KKT  system to the following  semi-definite optimization problem:
\begin{equation}\label{1.1}
\begin{array}{ll}
\min & f(x)\\
{\rm s.t.} & {\cal G}(x) \in {\cal K},
\end{array}
\end{equation}
where $f:{\cal X} \rightarrow\Re$ is a twice continuously differentiable  function, ${\cal G}:{\cal X}\rightarrow{\cal Y}$ is a twice continuously differentiable mapping with ${\cal G}=(\phi,\psi)$ for $\phi:{\cal X}\rightarrow \mathbb{S}^p$ and $\psi:{\cal X}\rightarrow {\cal Z}$,  ${\cal K}=\mathbb{S}^{p}_-\times {\cal P}$ and ${\cal P} \subset {\cal Z}$ is a nonempty convex polyhedral set.
Let $\Phi=\{x\in {\cal X}:{\cal G}(x)\in{\cal K}\}$ be the feasible set for problem (\ref{1.1}). Let $\bar{x}\in \Phi$. We say that Robinson's  constraint qualification (RCQ) for problem (\ref{1.1}) holds at $\bar{x}$ if
$$0\in {\rm int} \{ {\cal G}(\bar {x})+{\rm D}{\cal G}(\bar{x}){\cal X}-{\cal K}\},
$$
where ``int" denotes the topological interior part of a given set.
The Largangian function of  problem (\ref{1.1}) is defined as
$$
{\cal L}(x;y,z):=f(x)+ \langle y, \phi(x)\rangle+\langle z, \psi(x)\rangle, \quad \forall\, (x,y,z) \in \X\times {\mathbb{S}^p} \times {\cal Z}.
$$
For any $(y,z)\in \mathbb{S}^p\times {\cal Z}$, let ${\rm D}_x {\cal L}({x}; {y},z)$ denote  the derivative of ${\cal L}(\cdot; y,z)$ at $x\in {\cal X}$ and denote   $\nabla_x{\cal L}(x;y,z):=({\rm D}_x{\cal L}(x;y,z))^*$.
If   there exists $(\bar{y},\bar z)\in {\cal Y}$  such that $(\bar x, \bar{y}, \bar{z})$ satisfies  the KKT system
\begin{equation}\label{2.3}
\left\{
\begin{array}{l}
\nabla_x{\cal L}(\bar{x}; \bar{y},\bar z)=0,\\[4pt]
(\bar{y},\bar z)\in  N_{\cal K}({\cal G}(\bar{x})),
\end{array}
\right.
\end{equation}
then we call $\bar{x}$ a stationary point of problem (\ref{1.1}) and  $(\bar y,\bar z)$ a Lagrangian multiplier of problem (\ref{1.1}) at $\bar x$. Here $N_{\cal K}(w)$ denotes the normal cone of ${\cal K}$  at $w\in \Y$.   Denote by $\Lambda(\bar{x})$ the
  set of all $(\bar y,\bar z)\in  {\mathbb{S}^p} \times {\cal Z}$ satisfying (\ref{2.3}). If $\bar x$ is a local minimizer to  problem (\ref{1.1}), then the set $\Lambda(\bar{x})$ is nonempty, convex and compact if and only if the RCQ holds at $\bar{x}$. The strict Robinson constraint qualification (SRCQ for short) at $\bar{x}$ with respect to $(\bar y,\bar z)\in \Lambda(\bar{x})$ is defined by (see Bonnans and Shapiro \cite{BS2000})
\begin{equation}\label{2.5}
\begin{array}{l}
{\rm D} {\cal G} (\bar{ x}) \X + T_{\cal K}({\cal G}(\bar{ x}))\cap(\bar{ y},\bar z)^\bot = {\cal Y},
\end{array}
\end{equation}
where for any vector $w\in \Y$, $w^\perp: = \{ y\in \Y: \, \langle w, y\rangle =0\}$.
 Obviously, the SRCQ is more restrictive than the RCQ. It follows from  Bonnans and Shapiro \cite[Proposition 4.50]{BS2000} that the set of Lagrange multipliers $\Lambda(\bar{ x})$ is a singleton if the SRCQ
 (\ref{2.5}) holds. 

Let $\bar{ x}\in\Phi$ be a feasible point. The critical cone of problem (\ref{1.1}) at $\bar{ x}$ is defined by\\
$$
{\cal C}(\bar{ x}) := \{ d\in\X : \, {\rm D}{\cal G}(\bar{ x})d\in T_{\cal K} ({\cal G}(\bar{ x})),\ {\rm D}f(\bar x)d\leq 0\}.
$$
\begin{definition}[The second-order sufficient optimality condition]\label{Definition3.1}
Let $\bar{ x}$ be a stationary point of problem (\ref{1.1})
at which $\Lambda(\bar{ x}) \neq \emptyset$. We say that the second-order sufficient optimality condition for  problem (\ref{1.1}) holds at $\bar{ x}$ if
$$
\sup_{(y,z)\in \Lambda(\bar{ x})} \left\{ \left\langle d, \nabla_{xx}^2{\cal L}(\bar{x};y,z)d\right\rangle  + 2 \left\langle y, {\rm D}\phi(\bar x)d \left[-\phi(\bar x)\right]^\dagger {\rm D}\phi(\bar x)d\right \rangle\right\}> 0,\quad \forall\, 0\ne d\in {\cal C}(\bar{ x}),
$$
where  for   $(y,z) \in \mathbb{S}^p\times {\cal Z}$, $\nabla_{xx}^2{\cal L}(\cdot; y, z):={\rm D}_x[\nabla_x {\cal L}] (\cdot; y,z)$
  and for any matrix $S\in \mathbb{S}^p$, $S^\dagger$ denotes the Moore-Penrose pseudo-inverse of $S$.
\end{definition}
If follows from \cite[Theorem 3.86]{BS2000} that  if the second-order sufficient optimality condition for problem (\ref{1.1}) holds  at $\bar{ x}$, then the second-order growth condition for problem (\ref{1.1}) holds  at $\bar{ x}$, which implies that $\bar{x}$ is  a  strictly local optimal solution to problem (\ref{1.1}).

Define the KKT  mapping $G:{\cal X} \times\mathbb{S}^p\times  {\cal Z} \rightarrow {\cal X} \times \mathbb{S}^p\times  {\cal Z}$, associated with problem (\ref{1.1}), by
\begin{equation}\label{3.3}
G(x,y,z): =
\left[
\begin{array}{c}
\nabla_x{\cal L}(x;y,z)\\[6pt]
-{\cal G}(x)+ \Pi_{\cal K}({\cal G}(x)+(y,z))
\end{array}
\right],\quad \forall\, (x,y,z)\in {\cal X} \times\mathbb{S}^p\times  {\cal Z}.
\end{equation}
For characterizing  the isolated calmness  property for the mapping $G^{-1}$, we need some simple but useful properties on the non-polyhedral cone $\mathbb{S}^p_-$ and the polyhedral set ${\cal P}$.

Suppose that  $A\in \mathbb{S}^p$ and $ B\in \mathbb{S}^p_-$ are two matrices satisfying $A \in N_{{\mathbb{S}}^p_{-}}(B)$ or equivalently $B \in N_{{\mathbb{S}}^p_{+}}(A)$ with  $A\in \mathbb{S}^p_+$. Note that   $AB =BA=0$ and $B = \Pi_{\mathbb{S}^p_-}(B+A)$. Let   $C:= B+A$ and  $\lambda_1\ge \lambda_2 \ge \ldots\ge \lambda_p $ be its eigenvalues being arranged in the non-increasing order. Define $\alpha: =\{i: \, \lambda_i>0,  \, i=1, \ldots, p\}$,  $\beta :=\{i: \, \lambda_i=0, \, i=1, \ldots, p\}$ and
 $\gamma :=\{i: \, \lambda_i<0, \, i=1, \ldots, p\}$. Then there exists an orthogonal matrix $P\in \Re^{p\times p}$
 such that
 \begin{equation}\label{eq:spectralDecomposition}
 A = P\left[
\begin{array}{lll}
\Lambda_{\alpha}&0&0\\
0&0   &0\\
0&0&0_{\gamma}\\
\end{array}
\right]P^T, \ B = P\left[
\begin{array}{lll}
0_{\alpha}&0&0\\
0&0  &0\\
0&0&\Lambda_{\gamma}\\
\end{array}
\right]P^T,\ C = P\left[
\begin{array}{lll}
\Lambda_{\alpha}&0&0\\
0&0   &0\\
0&0&\Lambda_{\gamma}\\
\end{array}
\right]P^T,
 \end{equation}
where $\Lambda_\alpha\succ 0$ is  the diagonal matrix whose diagonal entries are $\lambda_i$ for $i\in \alpha$ and $\Lambda_{\gamma}\prec0$ is  the diagonal matrix whose diagonal entries are $\lambda_j$ for $j\in \gamma$, respectively.
 Write
$
P=
 [
P_{\alpha}\ P_{\beta} \ P_{\gamma}
 ]
$ with $P_\alpha \in \Re^{p\times |\alpha|}$, $P_\beta \in \Re^{p\times |\beta|}$ and $P_\gamma \in \Re^{p\times |\gamma|}$ and define
 $\Upsilon , \overline{\Upsilon} \in \Re^{|\alpha|\times |\gamma|}$   by
 $$
 \Upsilon_{ij}=\displaystyle \frac{-\lambda_j}{\lambda_i-\lambda_j},\quad \overline{\Upsilon}_{ij}=1-\Upsilon_{ij}, \quad \forall\, (i,j+|\alpha \cup \beta|) \in \alpha\times \gamma.
 $$
 It is known  from  \cite{BCShapiro98,BCShapiro99} that $\Pi_{\mathbb{S}^p_-}(\cdot)$ is directionally differentiable everywhere and   from  \cite{SSun02,PSSun2003} that
the directional derivative of  $\Pi_{\mathbb{S}^p_-}$ at $C$ along   $H\in \mathbb{S}^n$ is explicitly given by
\[\label{eq:extra100}
\Pi_{\mathbb{S}^p_-}^\prime(C;H)=\left
[
\begin{array}{ccc}
0 & 0 &P_{\alpha}^TH P_{\gamma}\circ \Upsilon\\[6pt]
0 & \Pi_{\mathbb{S}^{|\beta|}_-}(P_{\beta}^T H P_{\beta}) &P_{\beta}^T H P_{\gamma}\\[6pt]
P_{\gamma}^T H P_{\alpha}\circ \Upsilon^T & P_{\gamma}^T H P_{\beta} &P_{\gamma}^T H P_{\gamma}
\end{array}
\right ],
\]
where $``\circ"$ denotes the Hadamard product.
 Then, by Arnold \cite{Arnold71}, we know that the tangent cone of $\mathbb{S}^p_-$ at $ B\in \mathbb{S}^p_-$ takes the form of
$$
T_{\mathbb{S}^p_-}(B)=  \{H\in \mathbb{S}^p: \, H = \Pi_{\mathbb{S}^p_-}^\prime(B;H) \}=\{H\in \mathbb{S}^p:\, [P_\alpha\ P_\beta]^T H [P_\alpha\ P_\beta] \preceq 0\}
$$
and the critical cone of $\mathbb{S}^p_-$ at $C$, associated with $A \in N_{{\mathbb{S}}^p_{-}}(B)$,  is given by
\begin{equation}\label{eq:criticalCone}
{\cal C}_{\mathbb{S}^p_-}(C):= T_{\mathbb{S}^p_-}(B)\cap A^{\perp}=\left \{H\in \mathbb{S}^p:\, P_\alpha^T H [P_\alpha\ P_ \beta] =0, \ P_\beta^T H P_\beta\preceq 0\right\}.
\end{equation}
Analogously, the critical cone of $\mathbb{S}^p_+$ at $C$, associated with $B\in N_{{\mathbb{S}}^p_{+}}(A)$,  is given by
\begin{equation}\label{eq:criticalConeP}
 {\cal C}_{\mathbb{S}^p_+}(C): = T_{\mathbb{S}^p_+}(A)\cap B^{\perp}=\left \{H\in \mathbb{S}^p:\, P_\gamma^T H [P_\beta\ P_ \gamma] =0, \ P_\beta^T H P_\beta\succeq 0\right\}.
\end{equation}
\begin{lemma}\label{lem:projection-sigma-matrix} Suppose that $A\in \mathbb{S}^p$ and $ B\in \mathbb{S}^p_-$ are two matrices satisfying $A \in N_{{\mathbb{S}}^p_{-}}(B)$. Let  $A$, $B$ and $C:= B+A$ have the spectral decompositions as in (\ref{eq:spectralDecomposition}). Then we have the following results:
\begin{description}
\item[(i)]  For any given matrix $H\in \mathbb{S}^p$,
$$
H\in \left ({\cal C}_{\mathbb{S}^p_-}(C)\right)^{\circ}\Longleftrightarrow P_\alpha^THP_\gamma =0 \quad  \&\quad   H \in {\cal C}_{\mathbb{S}^p_+}(C)
$$
and
$$
H \in  \left(   {\cal C}_{\mathbb{S}^p_+}(C) \right)^{\circ} \Longleftrightarrow P_\alpha^THP_\gamma =0\quad \& \quad H\in  {\cal C}_{\mathbb{S}^p_-}(C).
$$

\item[(ii)] Let $\Delta A$ and $\Delta B$ be two matrices in $\mathbb{S}^p$. Then
$$
\Delta A-\Pi'_{\mathbb{S}^{p}_-}(C;\Delta A+\Delta B)=0
$$
if and only if
 \begin{equation}\label{eq:relationskL}
  \begin{array}{l}
   P^{T}_{\alpha}(\Delta A)[P_{\alpha} \ P_{\beta}]=0,\\
    P^{T}_{\alpha}(\Delta A)P_{\gamma} \circ  \overline{\Upsilon}
    =P^{T}_{\alpha}(\Delta B)P_{ \gamma}\circ \Upsilon,\\
  P^{T}_{\beta}(\Delta A)P_{\beta}=\Pi_{\mathbb{S}^{|\beta|}_-}(P^{T}_{\beta}(\Delta A+ \Delta B)P_{\beta}),\\
     {}[P_{\beta} \ P_{\gamma}]^T (\Delta B) P_{\gamma}=0.\\
     \end{array}
  \end{equation}
Moreover, the relations in (\ref{eq:relationskL}) imply
$$\Delta A\in {\cal C}_{\mathbb{S}^p_-}(C) \quad \& \quad
 \langle \Delta A,\Delta B \rangle=2\langle A,(\Delta A) [-B]^{\dag} (\Delta A)\rangle.
 $$
 \end{description}
\end{lemma}
\begin{proof} The conclusions  of  part (i) follow directly from (\ref{eq:criticalCone}) and (\ref{eq:criticalConeP})
 while
 the  conclusions of  part (ii) can be derived with no difficulty from (\ref{eq:spectralDecomposition}), (\ref{eq:extra100}), (\ref{eq:criticalCone})
 and the fact that
 $$P^{T}_{\beta}(\Delta A)P_{\beta}=\Pi_{\mathbb{S}^{|\beta|}_-}(P^{T}_{\beta}(\Delta A+ \Delta B)P_{\beta}) \Longleftrightarrow
{\mathbb{S}}^p_-  \ni P^{T}_{\beta}(\Delta A)P_{\beta}\perp P^{T}_{\beta}(\Delta B)P_{\beta}  \in {\mathbb{S}}^p_+.$$
We omit the details here.
 \end{proof}

 \begin{lemma}\label{lem:projection-sigmaP}
Let ${\cal P} \subset {\cal Z}$ be a given nonempty convex polyhedral set.
\begin{description}
\item[(i)] Let $a,b,c\in {\cal Z}$. Write $c_+:=\Pi_{{\cal P}}(c)$ and $c_-:=c-c_+$. Then
\begin{equation}\label{eq:a20i-1}
 \Pi'_{{\cal P}}(c; b)= \Pi_{T_{\cal P}(c_+)\cap c_-^{\bot}}(b).
\end{equation}
Moreover,
\begin{equation}\label{eq:a20i}
a -\Pi'_{{\cal P}}(c; a+b)=0
\end{equation}
if and only if
\begin{equation}\label{eq:poly-cri}
a\in T_{{\cal P}}(c_+)\cap c_-^{\bot} \quad \&\quad  b \in N_{T_{{\cal P}}(c_+)\cap c_-^{\bot}}(a).
\end{equation}

\item[(ii)] Let   $b \in {\cal P}$ and $0 \in a+ N_{{\cal P}}(b)$.  For the critical cone
   $$ {\cal C}_{{\cal P}}(b-a):=T_{{\cal P}}(b)\cap a^{\perp},
    $$
    we have
    \[\label{critical-expr}
    {\cal C}_{{\cal P}}(b-a)=S_{b,a}^{*},
    \]
    where     $S_{b,a}$ is a nonempty closed convex cone defined by
   $$ S_{b,a}:=\{u\in {\cal Z}:\, \langle u, b\rangle+(\delta^*_{{\cal P}})'(-a;-u)=0\}.
    $$
\end{description}
\end{lemma}
\begin{proof}
Since ${\cal P}$ is a nonempty convex polyhedron, we have from   Theorem 4.1.1 of \cite{FP2003} that (\ref{eq:a20i-1}) is true and equality (\ref{eq:a20i}) is equivalent to
$$a=\Pi_{T_{\cal P}(c_+)\cap c_-^{\bot}}(a+b),
$$
which  is equivalent to (\ref{eq:poly-cri}).  So the conclusions in  part (i) hold.

Now we turn to the proof of part (ii).
It follows from  \cite[Corollary 19.2.1]{Roc70} that $\delta^*_{\cal P}$ is a proper closed convex polyhedral function. Then we know from \cite[Theorem 23.10]{Roc70}  and \cite[Corollary 23.5.3]{Roc70} that
$$(\delta^*_{\cal P})'(-a;u)=\delta^*_{{\cal P}_{-a}}(u), \quad \forall\, u\in {\cal Z},
$$
where
$${\cal P}_{-a}:=\{b'\in {\cal P}:\langle -a,b'\rangle=\delta^*_{\cal P}(-a)\}.
$$
By using the assumption $-a \in N_{\cal P}(b) =\partial \delta _{\cal P} (b)$, we know that $b\in \partial \delta_{\cal P}^* (-a)$. Therefore, $b\in {\cal P}_{-a}$ and
$$
\begin{array}{ll}
S_{b,a}&=\{u: \langle u, b\rangle+(\delta^*_{\cal P})'(-a;-u)=0\}\\
&=\{u: \langle u, b\rangle+\delta^*_{{\cal P}_{-a}}(-u)=0\}\\
&=\{u: \langle u, b\rangle -\langle u, b'\rangle\leq 0,\  \forall \, b' \in {\cal P}_{-a}\}\\
&=\{u:\langle u, b'- b\rangle \geq 0,\ \forall\,  b' \in {\cal P}_{-a}\}\\
&=-N_{{\cal P}_{-a}}(b).
\end{array}
$$
Thus,  $S_{b,a} $ is a nonempty closed convex cone with $S_{b,a}^{*}=T_{{\cal P}_{-a}}(b)$.
  Since ${\cal P}_{-a}$ is a polyhedral set and ${\cal P}_{-a}= {\cal P} \cap L$, where $L:=\{b^\prime\in \Z: \, \langle b^\prime -b, a\rangle =0\}$, we have
$$
T_{{\cal P}_{-a}}(b) = T_{\cal P}(b)  \cap T_L (b) = T_{\cal P}(b) \cap a^{\perp}.
$$
Therefore,
$$
{\cal C}_{{\cal P}}(b-a) = T_{{\cal P}_{-a}}(b)= S_{b,a}^{*},
$$
which shows that (\ref{critical-expr}) holds.
The proof of this lemma is completed.
\end{proof}
\begin{lemma}\label{SRCQ-equi}
Let ${\cal A}:{\cal X} \rightarrow \mathbb{S}^p\times {\cal Z}$ be a linear operator and  $(c_1, c_2) \in \mathbb{S}^{p}\times \Z$.   Then  $v:=(v_1, v_2) \in \mathbb{S}^p\times {\cal Z}$ is a solution to  the following system of   equations
\begin{equation}\label{eq:a4xy}
\begin{array}{ll}
{\cal A}^*v &=0,\\[4pt]
\Pi'_{\mathbb{S}^{p}_-}(c_1;v_1) &=0,\\[4pt]
\Pi'_{{\cal P}}(c_2; v_2)&=0
\end{array}
\end{equation}
 if and only if
\begin{equation}\label{eq:a4xy000}
v \in \left[  {\cal A} \X + T_{\cal K}(c_+)\cap c_-^\bot\right ]^{\circ},
\end{equation}
where $c_+:=\Pi_{\cal K}(c)=(\Pi_{\mathbb{S}^{p}_-}(c_1),\Pi_{{\cal P}}(c_2))$ and $c_-=c-c_+$.
\end{lemma}
\begin{proof}
We have from   Lemma \ref{lem:projection-sigma-matrix} and (\ref{eq:criticalCone}) that
$$
\Pi'_{\mathbb{S}^{p}_-}(c_1;v_1) =0\Longleftrightarrow
v_1 \in [T_{\mathbb{S}^{p}_-}((c_1)_+)\cap (c_1)_-^{\bot}]^{\circ}.
$$
Since ${\cal P}$ is a convex polyhedron, we have from part (i) of  Lemma \ref{lem:projection-sigmaP} that
$$
\Pi'_{{\cal P}}(c_2; v_2)=0\Longleftrightarrow  v_2\in [ T_{{\cal P}}((c_2)_+)\cap (c_2)_-^{\bot}]^{\circ}.
$$
Thus,
$v$ satisfies  (\ref{eq:a4xy})    if and only if
$$
{\cal A}^* v= 0 \quad \& \quad  v\in  [ T_{\cal K}(c_+)\cap c_-^\bot]^{\circ}\ ,
$$
which is equivalent to saying that (\ref{eq:a4xy000}) holds.
The proof is completed.
\end{proof}


\begin{lemma}\label{lem:calmDirDer} Let $\bar{ x}\in\Phi$ be a stationary point of problem (\ref{1.1}) with $ (\bar y,\bar z)\in  \Lambda(\bar{ x}) \neq \emptyset$. Let the KKT mapping $G$ be defined by
(\ref{3.3}). Then   $G^{-1}$ is isolated  calm  at the origin for $(\bar x, \bar y,\bar z)$ if and only if $(d_x, d_y,d_z)=0$ for any $(d_x,d_y,d_z)\in {\cal X} \times \mathbb{S}^p \times {\cal Z}$ satisfying $ G'((\bar x, \bar y,\bar z);(d_x,d_y,d_z))=0$.
\end{lemma}
\begin{proof}
 By noting that $G$ is a locally Lipschitz continuous mapping around $(\bar x,\bar y,\bar z)$ and it is directionally differentiable at $(\bar x,\bar y,\bar z)$, we have for $(d_x,d_y,d_z)\in {\cal X} \times \mathbb{S}^p \times {\cal Z}$ that
$$\begin{array}{l}
 DG((\bar x,\bar y,\bar z)|0)(d_x,d_y,d_z) \\[6pt]
  =\left \{ \displaystyle \lim_{k \rightarrow \infty} \frac{G(\bar x+t_kd_x,\bar y+t_k d_y,\bar z+t_k d_z)-G(\bar x,\bar y,\bar z)}{t_k}\mbox{ for certain } t_k \searrow 0 \right \}\\[6pt]
 =\{G'((\bar x,\bar y,\bar z); (d_x,d_y,d_z))\}.
 \end{array}
$$
Thus, from   (\ref{inv-D}), we have for any $(d_x,d_y,d_z)\in {\cal X} \times \mathbb{S}^p \times {\cal Z}$ that
$$
(d_x,d_y,d_z)\in DG^{-1}(0|(\bar x,\bar y,\bar z))(0) \Longleftrightarrow
G'((\bar x, \bar y,\bar z);(d_x,d_y,d_z))=0,
$$
which, together with Lemma \ref{lem.1.1} and the fact that $G'((\bar x, \bar y,\bar z);(0,0,0))=0$, implies  that $G^{-1}$ is isolated  calm  at the origin for $(\bar x, \bar y,\bar z)$ if and only if
$$G'((\bar x, \bar y,\bar z);(d_x,d_y,d_z))=0 \Longrightarrow
 (d_x,d_y,d_z)=0.
$$
This completes the proof.
\end{proof}
\begin{theorem}\label{Theorem 3.1}
Let $\bar{ x}\in\Phi$ be a stationary point of problem (\ref{1.1}) with $ (\bar y,\bar z)\in  \Lambda(\bar{ x}) \neq \emptyset$. Then we have the following results:
 \begin{description}
 \item[(i)] If the   second-order sufficient optimality condition   for  problem (\ref{1.1}) holds at $\bar{ x}$ and the SRCQ (\ref{2.5}) holds at $\bar x$ with respect to $(\bar y,\bar z) $, then $G^{-1}$ is isolated  calm  at the origin for $(\bar x, \bar y,\bar z)$.
 \item[(ii)] If
$G^{-1}$ is isolated  calm  at the origin for $(\bar x, \bar y,\bar z)$,   then the SRCQ (\ref{2.5}) holds at $\bar x$ with respect to $(\bar y,\bar z)$.

         \item[(iii)]  If $G^{-1}$ is isolated  calm  at the origin for $(\bar x, \bar y,\bar z)$
         and  the quadratic form $$q:(d_x,d_x) \rightarrow \left\langle d_x, \nabla_{xx}^2{\cal L}(\bar{ x};\bar y,\bar z)d_x \right\rangle +2\left \langle \bar y, {\rm D}\phi(\bar x)d_x \left[-\phi(\bar x)\right]^\dagger {\rm D}\phi(\bar x)d_x \right\rangle$$
            satisfies $$q (d_x, d_x)  \ge 0,\   \forall\,   d_x \in {\cal C}(\bar x)\quad   \& \quad
            q(d_x, d_x)=0, \ d_x \in {\cal C}(\bar x)\Longrightarrow  \nabla_{xx}^2{\cal L}(\bar{ x};\bar y,\bar z)d_x=0,$$
            then the second-order sufficient optimality condition for  problem (\ref{1.1}) holds at $\bar{ x}$.
         \end{description}
\end{theorem}
\begin{proof} Since $ (\bar y,\bar z)\in  \Lambda(\bar{ x}) $, we know $\bar{y} \in N_{\mathbb{S}^p_{-}}(\phi(\bar x))$ and $\bar{z} \in N_{\cal P} (\psi(\bar x))$. Without loss of generality, we can assume that $A:=\bar y$, $B: =\phi(\bar{x})$ and $C := B+A$ have the spectral decompositions as in (\ref{eq:spectralDecomposition}).

We first prove part (i).
%
Let $(d_x,d_y,d_z)\in {\cal X} \times \mathbb{S}^p \times {\cal Z}$ be arbitrarily chosen such that  $G'((\bar x, \bar y,\bar z);(d_x,d_y,d_z))=0$.
Since the SRCQ (\ref{2.5}) holds at $\bar x$ with respect to $(\bar y,\bar z) $, we have from  \cite[Proposition 4.47]{BS2000} that the set of Lagrange multipliers of problem (\ref{1.1})  at $\bar x$ is a singleton, namely $\Lambda (\bar x)=\{(\bar y,\bar z)\}$. In this case, we can write the critical cone ${\cal C}(\bar x)$ as
$${\cal C}(\bar x)={\cal C}_1(\bar x)\cap  {\cal C}_2(\bar x),
$$
where
$$
\begin{array}{l}
{\cal C}_1(\bar x)=\{ d_x\in {\cal X}:\, {\rm D}\phi (\bar x)d_x \in T_{\mathbb{S}^{p}_-}(\phi (\bar x)),\ \langle \bar y,{\rm D}\phi (\bar x)d_x\rangle=0\},\\[4pt]
{\cal C}_2(\bar x)=\{ d_x\in {\cal X}:\, {\rm  D}\psi (\bar x)d_x \in T_{{\cal P}}(\psi(\bar x)), \ \langle \bar z,{\rm D}\psi (\bar x)d_x\rangle=0\}.\\[4pt]
  \end{array}
$$
Since  $G'((\bar x, \bar y,\bar z);(d_x,d_y,d_z))=0$,   we have
\begin{equation}\label{eq:a1}
\begin{array}{lr}
\nabla^2_{xx}{\cal L}(\bar x; \bar y,\bar z) d_x +{\rm D}{\cal G}(\bar x)^* (d_y,d_z) &=0,\\[4pt]
-{\rm D}{\cal G}(\bar x)d_x+\Pi'_{\cal K}({\cal G}(\bar x)+(\bar y,\bar z); {\rm D}{\cal G}(\bar x)d_x+(d_y,d_z))&=0.
\end{array}
\end{equation}
The second equation in  (\ref{eq:a1}) can be split into
$$\begin{array}{ll}
{\rm D}\phi(\bar x)d_x-\Pi'_{\mathbb{S}^{p}_-}(\phi(\bar x)+\bar y; {\rm D}\phi(\bar x)d_x+d_{y})&=0,\\
{\rm D}\psi(\bar x)d_x -\Pi'_{\cal P}(\psi(\bar x)+\bar z; {\rm D}\psi(\bar x)d_x+d_{z})&=0.
\end{array}
$$
Thus, we know from part (ii) of  Lemma \ref{lem:projection-sigma-matrix}  that
  $${\rm D}\phi(\bar x)d_x\in T_{\mathbb{S}^{p}_-}(\phi(\bar x))\cap {\bar y}^{\bot} \ \ \& \ \ \langle {\rm D}\phi(\bar x)d_x, d_{y}\rangle =2 \langle \bar y, {\rm D}\phi(\bar x)d_x \left[-\phi(\bar x)\right]^\dagger {\rm D}\phi(\bar x)d_x \rangle$$
and  from (i) of Lemma \ref{lem:projection-sigmaP} that
 $${\rm D}\psi(\bar x)d_x\in T_{{\cal P}}(\psi(\bar x))\cap \bar z^{\bot} \ \ \& \ \
 \langle d_{z}, {\rm D}\psi(\bar x)d_x \rangle=0.$$
Therefore, $d_x \in {\cal C}(\bar x)$. By taking the inner product between $d_x$ and both sides of the first equation in (\ref{eq:a1}), we obtain
$$
\left\langle d_x,\nabla^2_{xx}{\cal L}(\bar x, \bar y,\bar z) d_x\right\rangle +\left \langle d_x, {\rm D}{\cal G}(\bar x)^* (d_y,d_z) \right\rangle=0
$$
and thus
$$\left \langle d_x,\nabla^2_{xx}{\cal L}(\bar x, \bar y,\bar z) d_x\right\rangle +2\left\langle \bar y,{\rm D}\phi(\bar x)d_x \left[-\phi(\bar x)\right]^\dagger {\rm D}\phi(\bar x)d_x\right\rangle=0.
$$
It then follows from the second-order sufficient optimality condition    for  problem (\ref{1.1})   at $\bar x$  that $d_x=0$. Hence  (\ref{eq:a1}) is reduced to
$$\begin{array}{ll}
{\rm D}{\cal G}(\bar x)^*(d_y,d_z) &=0,\\[4pt]
\Pi'_{\mathbb{S}^{p}_-}(\phi(\bar x)+\bar y; d_{y})&=0,\\[4pt]
\Pi'_{{\cal P}}(\psi(\bar x)+\bar z; d_{z})&=0.
\end{array}
$$
In view of Lemma \ref{SRCQ-equi}, we obtain
$$
(d_y,d_z) \in \left[  {\rm D} {\cal G} (\bar{ x}) \X + T_{\cal K}({\cal G}(\bar{ x}))\cap (\bar{ y},\bar z)^\bot\right ]^{\circ},
$$
which  implies  $(d_y,d_z)=0$  from the assumed SRCQ (\ref{2.5}). Therefore,  $(d_x,d_y,d_z)=0$. Then, we know from Lemma \ref{lem:calmDirDer} that   $G^{-1}$ is isolated  calm  at the origin for $(\bar x, \bar y,\bar z)$.

 Now we prove   part (ii).   Suppose that the SRCQ (\ref{2.5}) does not hold at $\bar x$ for  $(\bar y,\bar z) \in \Lambda  (\bar x)$, namely
$$\begin{array}{l}
\Gamma:={\rm D} {\cal G} (\bar{ x}) \X + T_{\cal K}({\cal G}(\bar{ x}))\cap(\bar{ y},\bar z)^\bot \ne {\cal Y}.
\end{array}
$$
Then  there exists  $0\ne (\hat y,\hat{z}) \in \mathbb{S}^p\times {\cal Z}$ such that $(\hat y,\hat{z}) \in \Gamma^\circ$ or equivalently
$$
0 \ne (\hat y,\hat z) \in \left({\rm D} {\cal G} (\bar{ x})^*\right)^{\perp}\cap \left [ ( T_{\mathbb{S}^{p}_-}(\phi(\bar{ x}))\cap\bar{ y}^\bot )^{\circ} \times( T_{\cal P}(\psi(\bar{ x}))\cap\bar{ z}^\bot )^{\circ}\right ].
$$
 Then we have from Lemma \ref{SRCQ-equi} that
 $$\begin{array}{ll}
 {\rm D}{\cal G}(\bar x)^*(\hat y,\hat z) &=0,\\[4pt]
 \Pi'_{\mathbb{S}^{p}_-}(\phi(\bar x)+\bar y; \hat y) &=0,\\[4pt]
\Pi'_{\cal P}(\psi(\bar x)+\bar z;\hat z)&=0,\\[4pt]
\end{array}
  $$
 which imply
 $$
 G'((\bar x, \bar y,\bar z);(0,\hat y,\hat z))=0, 
 $$
 that is
$$
 0 \in DG((\bar x, \bar y,\bar z)|0)(0,\hat y,\hat z).
$$
  Since $G^{-1}$  is assumed to be isolated  calm  at the origin  for $(\bar x, \bar y,\bar z)$, we obtain from Lemma \ref{lem.1.1}  that $(  \hat y,\hat z)=0$. This   contradiction shows that the assertion in part (ii) is true.

 Finally, we prove  part (iii) by contradiction. Suppose that the second-order sufficient optimality condition for problem (\ref{1.1}) does not hold at $\bar{ x}$. Since $G^{-1}$ is assumed to be isolated  calm  at the origin for $(\bar x, \bar y,\bar z)$, we have $\Lambda(\bar x) =\{(\bar y, \bar z)\}$. Thus,
 there exists a vector $0\ne d_x \in {\cal C}(\bar x)$ satisfying $q(d_x,d_x)=0$.
 We then  know from the conditions  given  in part (iii) that $\nabla_{xx}^2 {\cal L}(\bar x; \bar y,\bar z)d_x=0$ and thus  $\big \langle \bar y, {\rm D}\phi(\bar x)d_x \left[-\phi(\bar x)\right]^\dagger {\rm D}\phi(\bar x)d_x \big \rangle=0$.
Moreover, from the definition of ${\cal C}(\bar x)$ and  Lemmas \ref{lem:projection-sigma-matrix} and \ref{lem:projection-sigmaP},  
 we have
  \begin{equation}\label{eq:a2pp}
\begin{array}{ll}
{\rm D}\phi(\bar x)d_x -\Pi'_{\mathbb{S}^{p}_-}(\phi(\bar x)+\bar y; {\rm D}\phi(\bar x)d_x)&=0,\\[4pt]
{\rm D}\psi(\bar x)d_x-\Pi'_{\cal P}(\psi(\bar x)+\bar z; {\rm D}\psi(\bar x)d_x)&=0.
\end{array}
\end{equation}
 By using $\nabla_{xx}^2 {\cal L}(\bar x; \bar y,\bar z)d_x=0$, (\ref{eq:a2pp}) and the expression of the directional derivative of $G$ at $(\bar x, \bar y,\bar z)$, we get  $G'((\bar x, \bar y,\bar z);(d_x,0,0))=0$ with $d_x \ne 0$. Then, by Lemma \ref{lem:calmDirDer}, we arrive at a contradiction with
   the isolated calmness of  $G^{-1}$   at the origin for $(\bar x, \bar y,\bar z)$. Therefore, we must have $q(d_x,d_x)>0$ for $d_x \in {\cal C}(\bar x)\setminus \{0\}$. That is,   the second-order sufficient optimality condition for problem (\ref{1.1}) holds at $\bar x$. The proof is completed.
\end{proof}

Based on Theorem \ref{Theorem 3.1}, for linearly  constrained convex optimization problems, we obtain the following complete characterization on the isolated calmness of $G^{-1}$.
\begin{corollary}\label{Coro-convex-case}
Let $f$ be a twice continuously differentiable convex function, ${\cal G}$ be an affine mapping and $\bar x$ be a minimizer to  problem (\ref{1.1}) with $\Lambda (\bar x)\ne \emptyset$.  Then $G^{-1}$ is isolated calm  at the origin for   $(\bar x, \bar y,\bar z)$ with $(\bar y,\bar z) \in \Lambda (\bar x)$ if and only if  the second-order sufficient optimality condition for  problem (\ref{1.1}) holds at $\bar x$ and  the SRCQ (\ref{2.5}) holds at $\bar x$   for  $(\bar y,\bar z)\in \Lambda(\bar{ x})$.

\end{corollary}

\section{Convex composite quadratic semi-definite programming}\label{section:QSDP}
\setcounter{equation}{0}
 In this section  we shall further study the isolated calmness for the following important convex composite quadratic  SDP:
\begin{equation}\label{eq:QSDP0}
\begin{array}{ll}
\min & \displaystyle \frac{1}{2}\langle x, {\cal Q} x\rangle+\langle c,x\rangle\\
{\rm s.t.} & {\cal A}x =b,\ \ x\in \mathbb{S}^p_+\cap \mathcal{P},
\end{array}
\end{equation}
where $c \in \mathbb{S}^p$,   $b \in \Re^m$, ${\mathcal Q}: \mathbb{S}^p \to \mathbb{S}^p$ is a self-adjoint positive semi-definite linear operator,  ${\cal A}: \mathbb{S}^p \to \Re^m$ is a linear operator  and $\mathcal {P} $ is a simple  nonempty convex polyhedral set in $\mathbb{S}^p$. As in Subsection \ref{ccqcp}, by introducing an additional variable $u\in \mathbb{S}^p$, we can rewrite problem (\ref{eq:QSDP0})  equivalently as
\begin{equation}\label{eq:QSDP}
\begin{array}{ll}
\min & \displaystyle \frac{1}{2}\langle x, {\cal Q} x\rangle+\langle c,x\rangle\\
{\rm s.t.} & {\cal A}x =b,\ \ x-u=0,\ \ x\in \mathbb{S}^p_+ , \ \ u\in {\cal P} .
\end{array}
\end{equation}
%
%
 Suppose that $(\bar{x}, \bar u)\in \mathbb{S}^p_+ \times {\cal P}$ is an optimal solution to the convex optimization  problem (\ref{eq:QSDP}). Note that  $\bar u= \bar x$. Let $\Lambda_P(\bar{x}, \bar {u})$, which may be an empty set, denote the  set of Lagrange multipliers $(s,y,z,v) \in \mathbb{S}^p\times \Re^m \times \mathbb{S}^p\times \mathbb{S}^p$  for problem (\ref{eq:QSDP}) at $(\bar{x}, \bar{u})$ such that $(\bar{x},\bar{u}, s,y,z,v)$ satisfies the following  KKT system
\[\label{eq:LS-P}
\left\{
\begin{array}{l}{\cal Q}\bar x+c-{\cal A}^*y-z-s=0, \ z+ v=0,
\\
b-{\cal A} \bar{x}=0, \  \bar{u} -\bar{x}=0, \  s \in  N_{\mathbb{S}^{p}_-}(-\bar x), \   v\in    N_{{\cal P}}(\bar u)\, .
\end{array}\right.\]
The KKT  mapping $G_{P}$,
associated with problem (\ref{eq:QSDP}), for any $(x,u,s, y,z,v)\in \mathbb{S}^p \times\mathbb{S}^p\times \mathbb{S}^p\times \Re^m\times \mathbb{S}^p\times \mathbb{S}^p$ is given by
\begin{equation}\label{3.3-p}
G_P(x,u,s,y,z,v) :=
\left[
\begin{array}{c}
{\cal Q} x-{\cal A}^*y-z-s+c
\\
z+v
\\
  x+\Pi_{{\mathbb{S}}^p_{-} }(-x+s)\\
  {\cal A} x-b\\
x-u
\\
- u+\Pi_{{\cal P} }(u+v)
\end{array}
\right]  .
\end{equation}
We also define the  reduced  KKT  mapping $F_{P}$,
associated with problem (\ref{eq:QSDP}), as follows: for any $(x,u, y,z)\in   \mathbb{S}^p\times \mathbb{S}^p\times \Re^m\times \mathbb{S}^p$,
\begin{equation}\label{3.3-pSim}
F_P(x,u,y,z) :=
\left[
\begin{array}{c}
  x+\Pi_{{\mathbb{S}}^p_{-} }(-x+{\cal Q} x-{\cal A}^*y-z +c)\\
- u+\Pi_{{\cal P} }(u-z)
\\
  {\cal A} x-b\\
  x-u
\end{array}
\right]  .
\end{equation}
By using Lemma \ref{lem:calmDirDer}, we can easily obtain the following equivalence on the isolated calmness property of $(G_P)^{-1}$ and $(F_P)^{-1}$.
\begin{proposition}\label{propCalm5.1}
Let   $(\bar{x},\bar{u}, \bar{s},\bar{y},\bar{z}, \bar{v})\in \mathbb{S}^p\times \mathbb{S}^p\times \mathbb{S}^p\times \Re^m \times  \mathbb{S}^p\times    \mathbb{S}^p$ be a solution to the KKT system (\ref{eq:LS-P}). Then    $(G_P)^{-1}$  is isolated calm at the origin with respect to $(\bar x, \bar u, \bar s, \bar y,   \bar z, \bar{v})$ if and only if $(F_P)^{-1}$  is isolated calm at the origin with respect to $(\bar{x}, \bar u, \bar y,\bar{z})$.
\end{proposition}

The critical cone of  problem (\ref{eq:QSDP}) at $({\bar x}, \bar{u})$ is given  by
$$
{\cal C}(\bar x, \bar{u})=\{(d_x,d_u) \in {\mathbb S}^p\times {\mathbb S}^p:\, {\cal A}d_x=0,\, d_u-d_x=0, \, d_x \in T_{\mathbb{S}^p_+}(\bar x) , \, d_u \in T_{\cal P}(\bar u),\, \langle {\cal Q}\overline x+c, d_x\rangle=0 \}.
$$
If $\Lambda_P(\bar{x}, \bar{u}) \neq \emptyset$, then for any $(s,y,z,v) \in \Lambda_P(\bar{x}, \bar{u})$,
\[\label{eq:Cst}
{\cal C}(\bar x, \bar{u})=\{(d_x,d_u) \in {\mathbb S}^p\times {\mathbb S}^p:\, {\cal A}d_x=0,\, d_u-d_x=0, \, d_x \in {\cal C}_{\mathbb{S}^p_+}(\bar x-s) , \, d_u \in {\cal C}_{\cal P}(\bar u-z) \}.
\]


The Lagrange dual of problem (\ref{eq:QSDP}) takes the form of
\begin{equation}\label{eq:QSDPD-1}
\begin{array}{ll}
\max &\displaystyle \inf_{x\in \mathbb{S}^p} \left\{\frac{1}{2}\langle x, {\cal Q} x\rangle +\langle v , x\rangle \right\}+\langle b, y\rangle-\delta^*_{\cal P}(-z)\\
{\rm s.t.} & s+ {\cal A}^*y+v+z=c,\ \  s\in \mathbb{S}^p_+ ,
\end{array}
\end{equation}
which is equivalent to
\begin{equation}\label{eq:QSDPD0}
\begin{array}{ll}
\max & \displaystyle\langle b, y\rangle- \frac{1}{2}\langle w, {\cal Q} w\rangle  -\delta^*_{\cal P}(-z)\\
{\rm s.t.} & s+{\cal A}^*y- {\cal Q}w+z=c,\\[4pt]
& s\in \mathbb{S}^p_+ ,\  w \in {\cal W},
\end{array}
\end{equation}
where ${\cal W}$ is any linear subspace in $\mathbb{S}^p$ that  contains
${\rm Range\,} {\cal Q}$,  e.g., ${\cal W}=\mathbb{S}^p$ or ${\cal W}={\rm Range\,} {\cal Q}$. 
%
By introducing an additional variable $t$, we can reformulate problem (\ref{eq:QSDPD0}) equivalently as
\begin{equation}\label{eq:QSDPD}
\begin{array}{ll}
\max &\langle b, y\rangle -\displaystyle \frac{1}{2}\langle w, {\cal Q} w\rangle  -t\\
{\rm s.t.} & s+{\cal A}^*y-{\cal Q}w+z=c,\\[4pt]
& s\in \mathbb{S}^p_+ ,\  w \in {\cal W}, \ (z,t) \in {\rm epi\,}\theta\, ,
\end{array}
\end{equation}
where
$$
\theta(z) :=\delta^*_{\cal P}(-z), \quad \forall\, z\in \mathbb{S}^p.
$$
Let $(\bar s, \bar y, \bar w, \bar z)\in \mathbb{S}^p\times \Re^m \times {\cal W}  \times \mathbb{S}^p$ be an optimal solution to problem (\ref{eq:QSDPD0}). Then, obviously, $(\bar s, \bar y, \bar w, \bar z, \theta(\bar z))$ is an optimal solution to problem (\ref{eq:QSDPD}).
We use  $\Lambda_D(\bar s, \bar y, \bar w, \bar z )$ to denote the corresponding   set of Lagrange multipliers   for problem (\ref{eq:QSDPD0}) at $( \bar s, \bar y, \bar w, \bar z)$, that is $x\in \Lambda_D(\bar s, \bar y, \bar w, \bar z )$ if and only if   $(\bar s, \bar y, \bar w, \bar z, {x})$ satisfies the following KKT system
\[\label{eq:LS-D}
 0 \in x+N_{\mathbb{S}^{p}_+}(\bar s), \ {\cal A}x -b=0,  \  {\cal Q}  \bar w -{\cal Q}  x=0,\  0 \in x+   \partial \theta(\bar z), \ c-{\bar s}-{\cal A}^*\bar{y}+{\cal Q}\bar {w}-\bar{z}=0.
\]
%
%
%
%
Thus, the KKT  mapping $F_{ {D}}$, associated with problem (\ref{eq:QSDPD0}), can be defined  for any $(s,y,w, z,x  )\in  \mathbb{S}^p\times \Re^m\times {\cal W}\times \mathbb{S}^p\times \mathbb{S}^p $ that
\begin{equation}\label{3.3-D}
F_{{D}}(s,y,w,z,x) :=
\left[
\begin{array}{c}
{\cal A}x-b
\\
{\cal Q}w-{\cal Q}x
\\
-s-{\cal A}^*y+{\cal Q} w-z+c
\\
  s+\Pi_{{\mathbb{S}}^p_{-} }(-s+x)\\
- z+{{\rm Pr}_{\theta }}(z-x)
\end{array}
\right]  .
\end{equation}
Note that  for any  $x\in \Lambda_D(\bar s, \bar y, \bar w, \bar z )$, it holds that
 $$0 \in x+   \partial \theta(\bar z) \Longleftrightarrow 0\in (x, 1) + N_{{\rm epi\,}\theta}(\bar z, \theta(\bar z)) \Longleftrightarrow (\bar z,\theta(\bar z)) = \Pi_{{\rm epi\,}\theta} ( (\bar z,\theta(\bar z))-(x,1)).
$$
Moreover, since $\theta: \mathbb{S}^p \to (-\infty, +\infty]$  is a  proper closed convex polyhedral function  \cite[Corollary 19.2.1]{Roc70}, we know from convex analysis \cite[Theorem 23.10]{Roc70} that
$$T_{{\rm epi\,}\theta}  (\bar z,\theta(\bar z)) = \left( N_{{\rm epi\,}\theta}  (\bar z,\theta(\bar z))\right)^{\circ}  =\{(u,t) \in \mathbb{S}^p \times \Re : \,   \theta^\prime(\bar z; u) \le t \}.
$$
Thus, for any $x\in \Lambda_D(\bar s, \bar y, \bar w, \bar z )$,
$$\begin{array}{lcl}
T_{{\rm epi\,}\theta}  (\bar z,\theta(\bar z)) \cap (x,1)^{\perp} &=&\{(u,t) \in \mathbb{S}^p \times \Re : \, t =\langle u, -x\rangle = \theta^\prime(\bar z; u)\}
\\  & =& \{(u,t) \in \mathbb{S}^p \times \Re : \, u\in S_{ {x}, \bar{z}}, \  t =\langle u, -x\rangle \},
\end{array}
$$
where  for any $(x,z)\in \mathbb{S}^p\times \mathbb{S}^{p}$,  the set $S_{ {x}, z}$ is defined by
\[
S_{ {x}, z}:= \{u\in \mathbb{S}^p: \,   \langle u,   {x}\rangle+ \theta^\prime( z; u) = 0  \}
=\{u\in \mathbb{S}^p: \,   \langle u,   {x}\rangle+ (\delta^*_{\cal P})^\prime( -z; -u) = 0  \}.
\]
\begin{lemma}\label{lem:epi-Dirderivative} Let $x\in \Lambda_D(\bar s, \bar y, \bar w, \bar z )\ne \emptyset$. Then for any $(\delta z, \delta t) \in    \mathbb{S}^p\times \Re$ and  $\delta x\in \mathbb{S}^p$,
$$
(\delta z,\delta t)= (\Pi_{{\rm epi\,}\theta}) ^\prime ( (\bar z-x,\theta(\bar z)-1);(\delta z-\delta x,\delta t))\Longleftrightarrow \delta z =( {\rm Pr}_\theta)^\prime(\bar{z}-x; \delta z-\delta x), \ \delta t =\langle \delta z, -x\rangle.
$$
\end{lemma}
\begin{proof}
 By using Lemma \ref{lem:projection-sigmaP}, we have
 $$
 (\delta z,\delta t)= (\Pi_{{\rm epi\,}\theta}) ^\prime ( (\bar z-x,\theta(\bar z)-1);(\delta z-\delta x,\delta t))\Longleftrightarrow
 -\delta x \in N_{S_{x,\bar z}}(\delta z ), \ \delta t =\langle \delta z, -x\rangle.
 $$
By noting that for any $v\in \mathbb{S}^p$,
$
{\rm Pr}_\theta(v) = v +\Pi_{\cal P}(-v),
$ we  know  from  Lemma \ref{lem:projection-sigmaP} that
$$
 \delta z =( {\rm Pr}_\theta)^\prime(\bar{z}-x; \delta z-\delta x)\Longleftrightarrow \delta z
 =(\delta z-\delta x)+\Pi_{\cal P}^\prime(x-\bar z; \delta x-\delta z) \Longleftrightarrow   -\delta z \in N_{S_{x,\bar z}^*}(\delta x ).
$$
The conclusion of this lemma then follows.
\end{proof}
The KKT  mapping $G_{ {D}}$,
associated with problem (\ref{eq:QSDPD}), for any $(s,y,w, (z,t),x,u, (v,\zeta))\in  \mathbb{S}^p\times \Re^m\times {\cal W}\times (\mathbb{S}^p\times \Re)\times \mathbb{S}^p\times \mathbb{S}^p\times (\mathbb{S}^p\times \Re)$ is given by
\begin{equation}\label{3.3-hatD}
G_{ {D}}(s,y,w,(z,t),x,u,(v,\zeta)) :=
\left[
\begin{array}{c}
x-u\\
{\cal A}x-b
\\
{\cal Q}w-{\cal Q}x
\\
(x,1)+(v,\zeta)
\\
-s-{\cal A}^*y+{\cal Q} w-z+c
\\
  s+\Pi_{{\mathbb{S}}^p_{-} }(-s+u)\\
- (z,t)+\Pi_{{\rm epi\,}\theta }((z,t)+(v,\zeta))
\end{array}
\right]  .
\end{equation}
By  using Lemmas \ref{lem:projection-sigmaP}, \ref{lem:calmDirDer} and \ref{lem:epi-Dirderivative}, we can   obtain with no difficulty the following equivalence on the isolated calmness property of  $(G_{ {D}} )^{-1}$ and  $(F_D)^{-1}$.
\begin{proposition}\label{propCalm5.2} Let  $(\bar s,\bar y,\bar w, (\bar z,\theta(\bar z)),\bar x,\bar u, (\bar{v},-1))\in  \mathbb{S}^p\times \Re^m\times {\cal W}\times (\mathbb{S}^p\times \Re)\times \mathbb{S}^p\times \mathbb{S}^p\times (\mathbb{S}^p\times \Re)$ be such that
$
G_{ {D}}  (\bar s,\bar y,\bar w, (\bar z,\theta(\bar z)),\bar x,\bar u, (\bar{v},-1))=0.
$
  Then  $(G_{ {D}} )^{-1}$ is isolated calm at the origin with respect to $(\bar s,\bar y,\bar w, (\bar z,\theta(\bar z)),\bar x,\bar u, (\bar{v},-1))$ if and only if $(F_D)^{-1}$  is isolated calm at the origin with respect to $(   \bar s, \bar y,   \bar{w}, \bar{z},\bar x)$.
\end{proposition}

Based on the equivalence between problem (\ref{eq:QSDPD}) and problem (\ref{eq:QSDPD0}), as   in \cite{ZSToh2010} for the linear SDP case,
     we can now introduce the concept of the extended SRCQ for problem (\ref{eq:QSDPD0}) in the following definition.

\begin{definition}\label{DefinitionECQ} Suppose that
 $\Lambda_P(\bar{x}, \bar u  ) \neq \emptyset$.
We say that  the extended SRCQ    for the dual problem  (\ref{eq:QSDPD0}) holds at $\Lambda (\bar{x}, \bar{u} )$  with respect to $(\bar{x},\bar{u})$  if
 \begin{equation}\label{eq:extendedSRCQD0}
  {\rm conv}\, \left \{ \bigcup \limits_{(s,y,z,v)\in \Lambda_P(\bar{x}, \bar{u})}  \left (  T_{\mathbb{S}^p_+}( s) \cap {\bar x}^{\perp} + S_{\bar{x}, z}\right)\right \}+{\cal A}^* \Re^m -{\cal Q} {\cal W}  =\mathbb{S}^p,
\end{equation}
where ``conv" denotes the convex hull of a set.
\end{definition}
Now we can establish  the relationship between the second-order sufficient optimality condition for problem  (\ref{eq:QSDP}) and  the extended SRCQ  for problem (\ref{eq:QSDPD0}).

\begin{proposition}\label{eqiv-SOSC-RCQ}
Let  $(\bar{x}, \bar u)\in \mathbb{S}^p _+\times {\cal P}$ be an optimal solution to  problem (\ref{eq:QSDP})
 with $\Lambda_P(\bar{x}, \bar{u}) \neq \emptyset$. Let   ${\cal W}\subseteq \mathbb{S}^p$ be any linear subspace that  contains
${\rm Range\,} {\cal Q}$. Then   the following two conditions are equivalent:
\begin{description}
\item[(i)]
The second-order sufficient optimality condition  for  the primal problem  (\ref{eq:QSDP}) holds at  $(\bar{x}, \bar{u})$:
\begin{equation}\label{eq:sosoc-primal}
\sup_{(s,y,z,v)\in \Lambda_P(\bar {x}, \bar{u})}\quad\left\{  \langle {\cal Q}d_x,d_x\rangle + 2 \langle s, d_{x}{\bar x}^\dagger  d_{x}\rangle\right\}> 0,\quad \forall\, 0\ne (d_x, d_u)\in {\cal C}(\bar{x}, \bar{u}).
\end{equation}

 \item[(ii)] The extended SRCQ  (\ref{eq:extendedSRCQD0}) for the dual  problem (\ref{eq:QSDPD0}) holds at  $\Lambda_P(\bar{x}, \bar{u})$  with respect to  $(\bar{x},\bar{u})$.
     \end{description}
\end{proposition}
\begin{proof} For notational convenience, denote
 $$
  \Gamma:=  {\rm conv}\, \left \{ \bigcup \limits_{(s,y,z,v)\in \Lambda_P(\bar{x}, \bar{u})}  \left (  T_{\mathbb{S}^p_+}( s) \cap {\bar x}^{\perp} + S_{\bar{x}, z}\right)\right \} \quad {\rm and} \quad
   {\cal D}: =\Gamma  +{\cal A}^* \Re^m-{\cal Q} {\cal W} .
$$

 $``{\rm (i)}\Longrightarrow {\rm (ii)}"$   We prove  this part by contradiction. Suppose that the extended SRCQ  (\ref{eq:extendedSRCQD0}) for the dual  problem (\ref{eq:QSDPD0})   does not hold at  $\Lambda( \bar x, \bar u)$ with respect to $(\bar{x},\bar{u})$.
  Then ${\cal D} \ne \mathbb{S}^p$. Let $ {\rm cl}({\cal D})$ denote the closure of ${\cal D}$. Since    ${\rm cl}({\cal D}) \ne \mathbb{S}^p$ (cf. \cite[Theorem 6.3]{Roc70}), there exists a point $a\in \mathbb{S}^p$ but $a\notin {\rm cl}({\cal D})$. Let $\bar h: =\Pi_{{\rm cl} ({\cal D})}(a)-a$.  By using the fact that ${\rm cl}({\cal D})$ is a closed convex cone in $\mathbb{S}^p$, we have
  $$
  \langle \bar h , d\rangle \ge 0, \quad \forall\, d \in  {\rm cl} ( {\cal D}) ,
  $$
  which, together with the assumption  ${\rm Range\,}{\cal Q} \subseteq {\cal W}$,  implies that ${\cal A} \bar h =0$, ${\cal Q} \bar h =0$ and
  \begin{equation}\label{eq:gamma}
   \langle \bar h, d\rangle \ge 0, \quad \forall\, d \in   \Gamma .
  \end{equation}
Let $(s, y,z,v)$ be an arbitrary point  in $\Lambda_P(\bar x, \bar u)$. Then $0 \in s+N_{\mathbb{S}^{p}_+}(\bar x)$ and  $ 0 \in z+   N_{{\cal P}}(\bar u)$. Since $\bar{x} \in N_{\mathbb{S}^{p}_-}(-s)$, without loss of generality, we can assume that $A:=\bar{x}$, $B := -s$ and $C:=-s+\bar{x}$ have the spectral decompositions as in
 (\ref{eq:spectralDecomposition}).
Then, by using (\ref{eq:gamma}),  part (i) of  Lemma \ref{lem:projection-sigma-matrix}
(applying to $A=\bar{x}$ and $B=-s$ and using $T_{\mathbb{S}^p_+} (s) = -T_{\mathbb{S}^p_-} (-s)$) and part (ii) of Lemma \ref{lem:projection-sigmaP} (applying to $a=z$ and  $b =\bar x$), we obtain (recall that $\bar{x}=\bar{u}$)
$$
\bar h \in {\cal C}_{\mathbb{S}^p_+} (\bar x-s), \quad
\langle s, \bar{h} {\bar x}^\dagger  \bar {h}\rangle=0\quad  \& \quad  \bar{h} \in {\cal C}_{\cal P} (\bar x-z).
$$
Therefore, $0\ne (\bar h, \bar h) \in {\cal C} (\bar x, \bar u)$.  Thus, by using the condition (\ref{eq:sosoc-primal}),  we know that  there exists
$(\bar{s}, \bar{y},\bar{z},\bar{v}) \in \Lambda_P(\bar x, \bar u)$ such that
$$
 \langle {\cal Q}\bar{h},\bar{h}\rangle + 2 \langle \bar {s}, \bar{h} {\bar x}^\dagger \bar{h}\rangle >0,
$$
which contradicts the proven   ${\cal Q}\bar{h}=0$ and $\langle \bar{s}, \bar{h} {\bar x}^\dagger  \bar {h}\rangle=0$. This contradiction shows that this  part holds.

 $``{\rm (ii)}\Longrightarrow {\rm (i)}"$   For the sake of contradiction we suppose that the second-order sufficient optimality condition (\ref{eq:sosoc-primal}) for  the primal problem  (\ref{eq:QSDP})  at  $(\bar{x}, \bar{u})$
 fails to hold. Then there exists $0\ne (\bar{h} , \bar{h}) \in {\cal C}(\bar x, \bar u)$ such that
$$\sup_{(s,y,z,v)\in \Lambda_P(\bar {x}, \bar{u})}\quad\left\{   \langle {\cal Q}\bar{h},\bar{h}\rangle + 2 \langle s, \bar{h}{\bar x}^\dagger  \bar{h}\rangle\right\}= 0,
 $$
 which implies
 $$
  \langle {\cal Q}\bar{h},\bar{h}\rangle =0\quad {\rm and} \quad \langle s, \bar{h}{\bar x}^\dagger  \bar{h}\rangle = 0, \quad \forall\,  (s,y,z,v)\in \Lambda_P(\bar {x}, \bar{u}) .
 $$
Let $(s, y,z,v)$ be an arbitrary point  in $\Lambda_P(\bar x, \bar u)$. By using the fact that $0 \in s+N_{\mathbb{S}^{p}_+}(\bar x)$ if and only if $\bar{x} \in N_{\mathbb{S}^{p}_-}(-s)$, without loss of generality, we can assume $A:=\bar{x}$, $B := -s$ and $C:=-s+\bar{x}$ have the spectral decompositions as in
 (\ref{eq:spectralDecomposition}). Then, from $\langle s, \bar{h}{\bar x}^\dagger  \bar{h}\rangle = 0$ we know that $P^T_\alpha \bar{h} P_\gamma =0$.

 Since the extended SRCQ  (\ref{eq:extendedSRCQD0}) is assumed to  hold, there exist $\hat y\in \Re^m$, $\hat w\in {\cal W}$ and $\hat d\in \Gamma$ such that $-\bar{h}= \hat d + {\cal A}^*\hat y-{\cal Q} \hat w $.
 By Carath{\'{e}}odory's theorem, there exist a positive  integer  $ k\le p(p+1)/2+1$,  scalars $\mu_i\ge 0$, $i=1, \ldots, k$, with $\mu_1 +\mu_2 +\ldots +\mu_k=1$, and points
 $$
 {\hat d}^{i} \in \bigcup \limits_{(s,y,z,v)\in \Lambda_P(\bar{x}, \bar{u})}  \left (  T_{\mathbb{S}^p_+}( s) \cap {\bar x}^{\perp} + S_{\bar{x}, z}\right), \quad i=1, \ldots, k
 $$
 such that $\hat d = \mu_1 {\hat d}^1   +\mu_2 {\hat d}^2+\ldots +\mu_k  {\hat d} ^k$. For each $\hat{d}^i$,
 there exist $(s^i,y^i, z^i,v^i) \in  \Lambda_P(\bar{x}, \bar{u})$,  $\hat{d}_1^i \in T_{\mathbb{S}^p_+}( s^i)\cap {\bar x}^{\perp}$ and  ${\hat d}_2^i \in S_{\bar{x}, z^i}$ such that ${\hat d}^i = {\hat d}_1^i +{\hat d}_2^i$. Then, by using
  ${\cal Q} \bar h=0$,  $P^T_\alpha \bar{h} P_\gamma =0$, $(\bar{h} , \bar{h}) \in {\cal C}(\bar x, \bar u)$, $T_{\mathbb{S}^p_+} (s) = -T_{\mathbb{S}^p_-} (-s)$, part (i) of  Lemma \ref{lem:projection-sigma-matrix} and part (ii)    of Lemma \ref{lem:projection-sigmaP}, we have
 $$
 \langle \bar{h}, \bar{h}\rangle =  \langle -\hat d- {\cal A}^*\hat y+{\cal Q} \hat w , \bar{h}\rangle =    \langle -\hat d, \bar{h}\rangle=-\sum_{i=1}^k \mu_i\langle {\hat d}_1^i +{\hat d}_2^i, {\bar h}\rangle   \le 0.
 $$
 This contradiction shows that this part is also true.
 \end{proof}

If  $\Lambda_P(\bar{x}, \bar{u})$ is a singleton, we have the following corollary.
\begin{corollary}\label{Coro-convex-caseP}
Let  $(\bar{x}, \bar u)\in \mathbb{S}^p_+ \times {\cal P}$ be an optimal solution to  problem (\ref{eq:QSDP}). If
  $\Lambda_P(\bar{x}, \bar{u})=\{(\bar s,\bar y,\bar z, \bar v)\}$, then   the following two conditions are equivalent:
\begin{description}
\item[(i)]
The second-order sufficient optimality condition  for  the primal problem  (\ref{eq:QSDP}) holds at  $(\bar{x}, \bar{u})$:
\begin{equation}\label{eq:sosoc-primalU}
  \langle {\cal Q}d_x,d_x\rangle + 2 \langle\bar{s}, d_{x}{\bar x}^\dagger  d_{x}\rangle > 0,\quad \forall\, 0\ne (d_x, d_u)\in {\cal C}(\bar{x}, \bar{u}).
\end{equation}

 \item[(ii)] The   SRCQ    for the dual  problem (\ref{eq:QSDPD0}) holds at  $ (\bar s,\bar y,\bar z,\bar v)$  with respect to  $(\bar{x}, \bar{u})$:
  \begin{equation}\label{eq:extendedSRCQD0U}
 T_{\mathbb{S}^p_+}     ( \bar{s}) \cap {\bar x}^{\perp} + S_{\bar{x}, \bar{z} }    +{\cal A}^* \Re^m-{\cal Q} {\cal W} =\mathbb{S}^p.
\end{equation}
     \end{description}
\end{corollary}
 In the next proposition, we shall establish an analogous  result to Proposition \ref{eqiv-SOSC-RCQ} between the second order sufficient optimization condition for   the dual  problem (\ref{eq:QSDPD0}) with  ${\cal W}={\rm Range\,} {\cal Q}$ and the extended SRCQ condition for the primal problem  (\ref{eq:QSDP}).

\begin{proposition}\label{eqiv-SOSC-RCQ2015}
  Let   ${\cal W}={\rm Range\,} {\cal Q}$ and $(\bar s, \bar y, \bar w, \bar z)\in \mathbb{S}^p\times\Re^m\times {\cal W}  \times \mathbb{S}^p$ be an optimal solution to the dual  problem (\ref{eq:QSDPD0})
 with $\Lambda_D(\bar s, \bar y, \bar w, \bar z ) \neq \emptyset$.  Then   the following two conditions are equivalent$:$
\begin{description}
\item[(i)]
The second-order sufficient optimality condition  for    the dual  problem (\ref{eq:QSDPD0}) holds at  $(\bar s, \bar y, \bar w, \bar z):$
\begin{equation}\label{eq:sosoc-dual}
\sup_{x\in \Lambda_D(\bar s, \bar y, \bar w, \bar z )}  \left\{  \langle {\cal Q}d_w,d_w\rangle + 2 \langle x, d_{s}{\bar s}^\dagger  d_{s}\rangle\right\}> 0, \ \forall\, 0\ne (d_s, d_y,d_w,d_z)\in {\cal C}(\bar s, \bar y, \bar w, \bar z ),
\end{equation}
where ${\cal C}(\bar s, \bar y, \bar w, \bar z)$ is the critical  cone consisting of all the vectors $ (d_s, d_y,d_w,d_z) \in \mathbb{S}^p\times \times \Re^m \times {\cal W} \times  \mathbb{S}^p$ such that
$$ d_s  +{\cal A}^* d_y-{\cal Q} d_w+ d_z=0, \quad  d_s \in T_{\mathbb{S}_+^p}(\bar s) \cap x^{\perp} \quad \&\quad  d_z \in S_{x,\bar s }.
$$
 \item[(ii)] The extended SRCQ      for the primal problem   (\ref{eq:QSDP})  holds at  $\Lambda_D(\bar s, \bar y, \bar w, \bar z )$  with respect to  $(\bar s, \bar y, \bar w, \bar z ):$
      \begin{equation}\label{eq:extendedSRCQPrimal}
  {\rm conv}\, \left \{ \bigcup \limits_{x\in \Lambda_D(\bar s, \bar w, \bar y, \bar z)} \left(  \left ( \begin{array}{c} {\cal A} \\ {\cal I} \end{array} \right) T_{\mathbb{S}^p_+}( x) \cap {\bar s}^{\perp}
  + \{0\} \times \left(T_{\cal P}({x}) \cap \bar{z}^{\perp}\right)\right) \right \}  =\Re^m\times \mathbb{S}^p.
\end{equation}
     \end{description}
\end{proposition}
\begin{proof}  Let
$$
\Gamma: ={\rm conv}\, \left \{ \bigcup \limits_{x\in \Lambda_D(\bar s, \bar y, \bar w, \bar z)} \left(  \left ( \begin{array}{c} {\cal A} \\ {\cal I} \end{array} \right) T_{\mathbb{S}^p_+}( x) \cap {\bar s}^{\perp}
  + \{0\} \times \left(T_{\cal P}({x}) \cap \bar{z}^{\perp}\right)\right) \right \} .
$$

$``{\rm (i)}\Longrightarrow {\rm (ii)}"$ Suppose that (\ref{eq:extendedSRCQPrimal}) does not hold. Then, by using the similar arguments as in the first part of the proof   for  Proposition \ref{eqiv-SOSC-RCQ}, we know that there exists
$0\ne \bar{h}= (\bar{h}_1, \bar{h}_2) \in \Re^m\times \mathbb{S}^p$ such that
$$
\langle  \bar{h}, d\rangle \ge 0, \quad \forall\, d\in \Gamma,
$$
which implies that for any $ x\in \Lambda_D(\bar s, \bar y, \bar w, \bar z)$,
\begin{equation}\label{eq:separationP}
{\cal A}^* \bar{h}_1 +\bar{h}_2 \in \left(T_{\mathbb{S}^p_+}(x) \cap {\bar s}^{\perp}\right)^* \quad \& \quad \bar{h}_2\in \left(T_{\cal P}({x}) \cap \bar{z}^{\perp}\right)^*.
\end{equation}

Let $x\in \Lambda_D(\bar s, \bar y, \bar w, \bar z)$ be fixed but arbitrarily chosen.  Then
 $0 \in x+N_{\mathbb{S}^{p}_+}(\bar s)$ and $0 \in x+   \partial \theta(\bar z)$. Since  $0 \in x+N_{\mathbb{S}^{p}_+}(\bar s)$ if and only if  $x\in N_{\mathbb{S}^{p}_-}(- s)$, we can assume that $A:= x$, $B: = -\bar{s}$ and $C:=-\bar{s}+x$ have the spectral decompositions as in
 (\ref{eq:spectralDecomposition}). Then we know from
 (\ref{eq:separationP}), part (i) of  Lemma \ref{lem:projection-sigma-matrix}
 and part (ii) of Lemma \ref{lem:projection-sigmaP} that
 $$
 P_\alpha ^T \left({\cal A}^* \bar{h}_1 +\bar{h}_2\right)P_\gamma =0, \quad {\cal A}^* \bar{h}_1 +\bar{h}_2\in  -({\cal C}_{\mathbb{S}_-^p} (-\bar s+x) )=T_{\mathbb {S}_+^p}(\bar s) \cap {x}^{\perp}  \quad \& \quad \bar{h}_2\in S_{x, \bar{z}}.
 $$
Let $d_s = -({\cal A}^* \bar{h}_1 +\bar{h}_2)$, $d_w =0\in {\cal W}$, $d_y =\bar h_1$ and $d_z =\bar h_2$. Then
we have
$$0\ne (d_s, d_y, d_w, d_z) \in {\cal C}(\bar s, \bar y, \bar w, \bar z )\quad \& \quad \langle x, d_{s}{\bar s}^\dagger  d_{s}\rangle=0,$$
which contradicts (\ref{eq:sosoc-dual}). This   completes the proof of ${\rm (i)}\Longrightarrow {\rm (ii)}$.

$``{\rm (ii)}\Longrightarrow {\rm (i)}"$ For the sake of contradiction   suppose that the second-order sufficient optimality condition (\ref{eq:sosoc-dual}) for  the dual  problem (\ref{eq:QSDPD0})  at  $(\bar s, \bar w, \bar y, \bar z)$ does not hold. Then  there exists $0\ne (d_s ,d_y, d_w, d_z  ) \in {\cal C}(\bar s, \bar y, \bar w, \bar z )$ such that
$$
\sup_{x\in \Lambda_D(\bar s, \bar y, \bar w, \bar z )}  \left\{   \langle {\cal Q}d_w,d_w\rangle + 2 \langle x, d_{s}{\bar s}^\dagger  d_{s}\rangle\right\}=0,
$$
which implies
$$
\langle {\cal Q}d_w,d_w\rangle=0 \quad \& \quad \langle x, d_{s}{\bar s}^\dagger  d_{s}\rangle=0, \quad \forall\, x\in \Lambda_D(\bar s, \bar y, \bar w, \bar z ).
$$
By using the fact that $d_w\in {\rm Range\,}{\cal Q}$, we know that $d_w=0$. Then, by mimicking the proof for  the second part of Proposition \ref{eqiv-SOSC-RCQ}, we can show that $d_s=0$, $d_y=0$ and $d_z=0$ and reach a contradiction to complete the proof of this part. The details are omitted here.
\end{proof}

If  $\Lambda_D(\bar s, \bar w, \bar y, \bar z )$ happens to be  a singleton, we have the following corollary.
\begin{corollary}\label{Coro-convex-caseD}
  Let   ${\cal W}={\rm Range\,} {\cal Q}$ and $(\bar s, \bar y, \bar w, \bar z)\in \mathbb{S}^p_+\times {\cal W} \times \Re^m \times \mathbb{S}^p_+$ be an optimal solution to the dual  problem (\ref{eq:QSDPD0})
 with $\Lambda_D(\bar s, \bar y, \bar w, \bar z ) = \{ \bar{x}\}$.  Then   the following two conditions are equivalent$:$
\begin{description}
\item[(i)]
The second-order sufficient optimality condition  for    the dual  problem (\ref{eq:QSDPD0}) holds at  $(\bar s, \bar y, \bar w, \bar z):$
\begin{equation}\label{eq:sosoc-dualU}
  \langle {\cal Q}d_w,d_w\rangle + 2 \langle\bar{ x}, d_{s}{\bar s}^\dagger  d_{s}\rangle  > 0, \ \forall\, 0\ne (d_s, d_y,d_w,d_z)\in {\cal C}(\bar s, \bar y, \bar w, \bar z ).
\end{equation}
 \item[(ii)] The   SRCQ      for the primal problem   (\ref{eq:QSDP})  holds at  $   \bar x  $  with respect to  $(\bar s, \bar y, \bar w, \bar z):$
      \begin{equation}\label{eq:extendedSRCQPrimalU}
  \left ( \begin{array}{c} {\cal A} \\ {\cal I} \end{array} \right) T_{\mathbb{S}^p_+}( \bar{x}) \cap {\bar s}^{\perp}
  + \{0\}  \times \left(T_{\cal P}(\bar{x}) \cap \bar{z}^{\perp}\right)   =\Re^m\times \mathbb{S}^p.
\end{equation}
     \end{description}
\end{corollary}

By noting that $(\bar{x},\bar{u}, \bar{s},\bar{y},\bar{z},\bar{v})$ is a solution to the KKT system (\ref{eq:LS-P}), i.e.,
$ G_P(\bar{x},\bar{u}, \bar{s},\bar{y},\bar{z},\bar{v})=0$, if and only if $F_{ {D}}  (\bar s,\bar y,\bar w,  \bar z,\bar{x})=0$ for some $\bar w\in {\cal W}$ such that ${\cal Q} \bar w={\cal Q}\bar{x}$,
we can   now state our main theorem of this section.

\begin{theorem}\label{eqiv-calmness} Suppose that  ${\cal W}={\rm Range\,} {\cal Q}$. Let $(\bar{x},\bar{u}, \bar{s},\bar{y},\bar{z},\bar{v})\in \mathbb{S}^p\times \mathbb{S}^p\times \mathbb{S}^p\times \Re^m \times  \mathbb{S}^p\times  \mathbb{S}^p$ be such that $ G_P(\bar{x},\bar{u}, \bar{s},\bar{y},\bar{z},\bar{v})=0$
 and $\bar{w}$ be the unique point in $  {\cal W}$   such that ${\cal Q} \bar{x} ={\cal Q} \bar{w}$.
Then  the following statements are equivalent to each other:
\begin{description}
\item[(i)] The second order sufficient optimality condition (\ref{eq:sosoc-primalU}) for the primal problem (\ref{eq:QSDP}) holds at $(\bar{x}, \bar{u})$  and the second order sufficient optimality condition (\ref{eq:sosoc-dualU})  for  the dual  problem (\ref{eq:QSDPD0})  holds at  $(\bar s, \bar y, \bar w, \bar z)$.
    \item[(ii)]     The   SRCQ    condition (\ref{eq:extendedSRCQPrimalU}) for the primal problem   (\ref{eq:QSDP})  holds at  $ \bar x$  with respect to  $(\bar s, \bar y, \bar w, \bar z )$
       and the     SRCQ  condition (\ref{eq:extendedSRCQD0U})  for the dual  problem (\ref{eq:QSDPD0}) holds at  $ (\bar s,\bar y,\bar z,\bar v)$  with respect to  $(\bar{x},\bar{u})$.
\item[(iii)] $(G_P)^{-1}$  is isolated calm at the origin with respect to $(\bar x, \bar u, \bar s, \bar y, \bar w, \bar z, \bar{v})$.
\item[(iv)]   $(G_{ {D}} )^{-1}$ is isolated calm at the origin with respect to $(\bar s,\bar y,\bar w, (\bar z,\theta(\bar z)),\bar x,\bar u, (-\bar{x},-1))$.

 \item[(v)]   $(F_P)^{-1}$  is isolated calm at the origin with respect to $(\bar{x}, \bar u, \bar y,   \bar z)$.

     \item[(vi)]   $(F_D)^{-1}$  is isolated calm at the origin with respect to $( \bar s, \bar y, \bar w, \bar z,\bar x)$.
    \item[(vii)]  The second order sufficient optimality condition (\ref{eq:sosoc-primalU}) for the primal problem (\ref{eq:QSDP}) holds at $(\bar{x},\bar{u})$  and the  SRCQ    condition (\ref{eq:extendedSRCQPrimalU}) for the primal problem   (\ref{eq:QSDP})  holds at  $ \bar x $  with respect to  $(\bar s, \bar y, \bar w, \bar z )$.
    \item[(viii)] The second order sufficient optimality condition  (\ref{eq:sosoc-dualU})   for  the dual  problem (\ref{eq:QSDPD0})   holds  at  $(\bar s, \bar y, \bar w, \bar z)$ and the     SRCQ  condition (\ref{eq:extendedSRCQD0U})  for the dual  problem (\ref{eq:QSDPD0}) holds at  $ (\bar s,\bar y,\bar w, \bar z)$  with respect to  $(\bar{x},\bar{u})$.
\end{description}
\end{theorem}
\begin{proof} By using the fact  that the conditions in  either (i) or (ii) or (vii) or (viii) imply that  $\Lambda_P(\bar{x}, \bar{u})=\{(\bar s,\bar y,\bar z, \bar v)\}$ and  $\Lambda_D(\bar s, \bar y, \bar w, \bar z)=\{\bar x\}$, we obtain  from Corollaries \ref{Coro-convex-caseP} and \ref{Coro-convex-caseD} that
$${\rm (i)} \Longleftrightarrow {\rm (ii)} \Longleftrightarrow {\rm (vii)} \Longleftrightarrow {\rm (viii)}.$$  {}By using  Lemma \ref{lem:projection-sigmaP} and the assumption that ${\cal W} = {\rm Range\,}{\cal Q}$, we can
obtain (v)$\Longleftrightarrow$ (vi) (refer to the proof of Lemma \ref{lem:epi-Dirderivative}).  Thus, by further using Propositions \ref{propCalm5.1} and \ref{propCalm5.2}, we have that the statements (iii)-(vi) are all equivalent to each other.  Finally, by noting  from Corollary \ref{Coro-convex-case} that  (iii) $\Longleftrightarrow$ (vii), we complete the proof.
\end{proof}

Recall that in Theorem \ref{th-1-rate} for the linear convergence rate of the sPADMM, we need Assumption \ref{erB2}. This assumption holds for problem \eqref{eq:QSDP} and its dual \eqref{eq:QSDPD0} if any one of the eight statements
in Theorem \ref{eqiv-calmness} is satisfied. Although Theorem \ref{eqiv-calmness} is only developed for convex composite  quadratic SDP, it is possible to extend it to other convex conic optimization  problems with the positive semi-definite cone being replaced by   some other   non-polyhedral but nice cones such as the second order cone or any finite Cartesian product of the  second order cones and the positive semi-definite cones.

\section{Conclusions}\label{section:final}
In this paper, we have provided a roadmap for the linear rate convergence of the sPADMM for solving linearly constrained convex composite optimization problems.  One significant feature of our approach relies on  neither the  strong convexity nor the strict complementarity. Our linear rate convergence  analysis for the convex composite quadratic programming is quite complete while   significant progress in convex nonlinear semi-definite programming, in particular in convex composite quadratic semi-definite programming, has been achieved.
 Perhaps, the most important   issue left unanswered is to provide  error bound results under weaker conditions for (convex) composite   optimization problems with non-polyhedral cone constraints. Another important   issue is to develop similar results for the inexact version  of the sPADMM, which is often more useful in practice.  However,  given the recent progress made on the inexact symmetric Gauss-Seidel based sPADMM in \cite{CSToh2015}, it does  not  seem to be difficult to extend our analysis to the inexact sPADMM.

\section*{Acknowledgements} The authors would like to thank Ying Cui and Xudong Li at National University of Singapore and Chao Ding at Chinese Academy of Sciences for their comments on an earlier version of this paper.


\end{document}